\documentclass[11pt]{article}
\usepackage{graphicx} % Required for inserting images
\usepackage[a4paper,top=2cm,bottom=2cm,left=3cm,right=3cm,marginparwidth=1.75cm]{geometry}

\usepackage{xcolor}

\usepackage[utf8]{inputenc}
\usepackage[english]{babel}
\usepackage{amsfonts}
\usepackage{amsmath,amsthm,amscd,amssymb,mathrsfs,setspace}
\usepackage{dsfont}
\usepackage{mathtools}

\usepackage[colorlinks,citecolor=lccx]{hyperref}
\usepackage[capitalize,nameinlink]{cleveref}

\usepackage{threeparttable}
\usepackage{tikz}

\usepackage{csquotes}

\usepackage{cancel}
\usepackage{float}
\usepackage{enumitem} 

\usepackage{booktabs}
\usepackage{multirow}
\usepackage{siunitx}
\usepackage{standalone}
\usepackage{pgfplots}
\usepgfplotslibrary{colormaps}
\usetikzlibrary{patterns}
\pgfplotsset{compat=1.18}
\newcommand\xl{-3}
\newcommand\yl{-3}
\newcommand\xu{3}
\newcommand\yu{3}
\newcommand\bluenocapity{0.1}
\newcommand\greenocapity{0.4 }
\newcommand\greynocapity{0.15}
\newcommand\orangeocapity{0.4 }
\newcommand\bound{2}
\newcommand\boundbound{4}
\newcommand\xlo{-\bound}
\newcommand\ylo{-\bound}
\newcommand\xuo{\bound}
\newcommand\yuo{\bound}
\newcommand\xloylu{-\boundbound}
\newcommand\xluylu{\boundbound}
\usepackage{authblk}

\usepackage{subcaption}

\usepackage{orcidlink}

\newcommand{\Ln}{{L^2(\Omega, \R^n)}}
\renewcommand{\L}{{L^2(\Omega)}}
\newcommand{\Hz}{H^1_0(\Omega)}

\newcommand{\C}{\mathbb{C}}
\newcommand{\R}{\mathbb{R}}
\newcommand{\N}{\mathbb{N}}
\newcommand{\K}{\mathbb{K}}
\newcommand{\1}{\mathds{1}}
\newcommand{\calQ}{\mathcal{Q}}
\newcommand{\calF}{\mathcal{F}}
\newcommand{\Ha}{\mathcal{H}}

\newcommand{\wad}{W_{\text{ad}}}
\newcommand{\tr}{\mathrm{tr}}
\renewcommand{\div}{\operatorname{div}}
\newcommand{\Hzerodiv}{H_0^{\mathrm{div}}}
\newcommand{\conj}[1]{\overline{#1}}

\DeclareMathOperator*{\sgn}{sgn}
\DeclareMathOperator*{\meas}{meas}

\DeclareMathOperator*{\TV}{TV}

\newcommand{\TVt}{\operatorname*{TV}_{\tau}}
\newcommand{\RT}{\operatorname*{RT0}_0^\tau}
\DeclareMathOperator*{\BV}{BV}
\DeclareMathOperator*{\argmin}{\arg\min}

\newtheorem{theorem}{Theorem}[section]
\newtheorem{proposition}[theorem]{Proposition}
\newtheorem{lemma}[theorem]{Lemma}
\newtheorem{corollary}[theorem]{Corollary}

\newtheorem{assumption}[theorem]{Assumption}

\theoremstyle{remark}
\newtheorem{remark}[theorem]{Remark}

\crefname{assumption}{Assumption}{Assumptions}
\Crefname{assumption}{Assumption}{Assumptions}

\makeatletter
\newcommand{\oset}[3][0ex]{%
  \mathrel{\mathop{#3}\limits^{
    \vbox to#1{\kern-2\ex@
    \hbox{$\scriptstyle#2$}\vss}}}}
\makeatother

\usepackage{marginnote}

\title{McCormick relaxations for PDE-constrained optimization on multi-dimensional domains}

\definecolor{lccx}{HTML}{92268F}
\ifpdf
\hypersetup{
    pdfauthor={M. Pokotylo and P. Manns},
	linkcolor=lccx
}
\fi
\author{Mariia Pokotylo $^{1}$\orcidlink{0009-0007-6495-8399}}
\author{Paul Manns $^{2}$\orcidlink{0000-0003-0654-6613}}

\affil{$^{1}$ Department of Mathematics, TU Dortmund University, Dortmund, Germany\\
\textit{mariia.pokotylo@tu-dortmund.de, paul.manns@tu-dortmund.de}}
\begin{document}
\maketitle
\begingroup
\renewcommand\thefootnote{\fnsymbol{footnote}}
\footnotetext[0]{
Mariia Pokotylo and Paul Manns acknowledge funding by Deutsche Forschungsgemeinschaft (DFG) under project no.~540198933. The authors thank Julian Hall for support with HiGHS and providing
	feedback on our instances in particular with respect to 
numerical difficulties of the solvers.}
\endgroup
\begin{abstract}
We are interested in computing approximate dual bounds for nonconvex optimal control problems using McCormick inequalities.
To handle the numerical burden concerned with the tightening of these inequalities, we locally average the bilinear term
in the PDE constraint. We extend recent theory in this area to elliptic PDEs on multidimensional domains
that involve bilinear terms. As an additional result, we also generalize the necessary assumptions to allow for more
flexibility in the choice of the discretization for the total variation penalty term from the underlying control problem. 
Furthermore, we prove the existence of valid bounds on the state variable, which are crucial to enable
the optimization-based bound-tightening (OBBT) procedure to tighten the McCormick inequalities.

Since theoretical bounds on the state variable are rather conservative, OBBT is both crucial and the main computational
bottleneck for deriving a sharp approximate dual bound. We show how the sequence of LPs arising during OBBT may
sensibly be warmstarted, allowing us to successfully scale the numerical experiments to finer discretizations
than what was possible before.

\end{abstract}
\section{Introduction}\label{sec:intro}
While there is vast literature about local optimization, first- and second-order 
algorithms, and their convergence to stationary points and local minimizers in 
optimal control of PDEs, literature on global optimization of such
problems is scarce if it is addressed at all with \cite{habeck2019global}
being a notable exception.
In this work, we contribute to this compute-intensive challenge and
concern ourselves with the task of computing approximate dual bounds
for nonconvex optimal control problems. Dual bounds are crucial for 
global optimization algorithms like (spatial) branch-\&-bound
\cite{land1960doig,tawarmalani2013convexification} to certify the 
(approximate) global optimality of a computed solution and prune parts of 
the decision space. To this end, we apply and extend the theory presented 
in \cite{kaltenbacher2026locally,manns} to elliptic control problems on 
multi-dimensional domains and provide some computational results
that demonstrate advantages as well as persisting challenges of
our approach. Specifically, we carry out our analysis
for two nonconvex model optimal control problems that are governed by elliptic PDEs that involve bilinear terms and have a total variation penalty term in the  objective.
	
Our first model problem is a multi-dimensional generalization of the 
one-dimensional academic model problem from \cite[\S3]{manns} and reads
\begin{gather} \label{eq : ocpe}
\begin{aligned}
\min_{u,w} \ &j(u,w) + \alpha \TV(w)\\
\text{s.t. } & \int_\Omega A \nabla u \cdot \nabla v + \int_\Omega  uwv = \int_\Omega fv \text{ for all } v \in \Hz \\
&w \in \wad, u \in \Hz,
\end{aligned} \tag{OCP$_E$}
\end{gather}
where $\Omega$ is a bounded domain with a polygonal boundary of dimension $n \in \{1,2,3\}$ and the set of admissible controls is $\wad \coloneqq \{w \in L^\infty(\Omega) : w_\ell \leq w \leq w_u \text{ a.e.}\}$.
The objective functional $j$ is Lipschitz continuous on bounded sets,
the source term $f \in \L$ is fixed, the operator
$ A \coloneqq (a_{ij})_{1 \leq i \leq n, 1 \leq j \leq n}$ is coercive
with a positive coercivity constant $\beta$ and bounded entries, that is
\begin{align} 
\beta |v|^2 &\le A v \cdot v \text{ for all } v \in \R^n \text{ and a.e.\ in } \Omega,
\label{eq : coerce_ocpe}\\
a_{ij} &\in L^\infty(\Omega) \text{ for all } i,j \in \{1, \ldots, n\}.\label{eq : bounded_ocpe}
\end{align}
In order to ensure ellipticity, we assume that $\beta$ satisfies $-\frac{\beta}{c_p^2} < w_\ell < w_u <\frac{\beta}{c_p^2}$,
where $c_p$ is a Poincaré constant depending on $\Omega$ and $w_\ell$, $w_u$ are some real fixed constants.
The constant $\alpha > 0 $ scales the regularization term $\TV(w)$, which denotes the total
variation of the control function $w$. 

The second model problem we consider leans on the
topology optimization problem from 
\cite{haslinger_topology_2015,leyffer2021convergence}.
The underlying PDE is a Helmholtz equation on a two-dimensional
square with Robin boundary conditions. It has a complex state space,
the Sobolev space $H^1(\Omega, \C)$. Specifically, the problem reads
\begin{gather} \label{eq : ocpc}
\begin{aligned}
\min_{u,w} \ &j(u,w) + \alpha \TV(w)\\
\text{s.t. } & \int_{\Omega} \nabla u \cdot \nabla \bar{v} - k_0^2 \int_{\Omega} (1 + qw) u \bar{v} - 
                            ik_0 \int_{\partial \Omega} u \bar{v} \\
                             &\quad\quad = k_0^2 \int_{\Omega} qw u_0 \bar{v}  \text{ for all } v \in H^1(\Omega, \C)\\
             &w \in C, u \in H^1(\Omega, \C)
\end{aligned} \tag{OCP$_H$}
\end{gather}
with $u_0 = \exp(ik_0 d \cdot x)$, 
$d \in \R^2$, $w \in L^\infty(\Omega)$, $q = \tilde{q}1_{\tilde{\Omega}}$
with $\tilde{q} > 0$ and $1_{\tilde{\Omega}}$ being the $\{0,1\}$-valued indicator function of $\tilde{\Omega} \subset \Omega$,
and a small enough wave number $k_0 > 0$. 
A sufficient condition
for $k_0$ is given later in
\cref{prp:helmholtz_existence_uniqueness}.
Note that with a
slight abuse of notation we apply the necessary trace operators in boundary integrals.

To obtain dual bounds, that is lower bounds, on \eqref{eq : ocpe} and
\eqref{eq : ocpc}, a convexification of the non-linear PDE is necessary.
A widely-used tool from finite-dimensional optimization are
McCormick relaxations \cite{McCormick}
that replace the bilinear term $uw$ by a linear one that itself is 
constrained by further linear McCormick inequalities which in turn depend
on $L^\infty(\Omega)$-bounds of $u$ and $w$. The resulting convex problem 
with only linear constraints is the so-called the McCormick relaxation. 
Importantly, the tightness of the McCormick relaxation hinges on the 
$L^\infty$-bounds on $u$ and $w$. While bounds on the control $w$
are generally given, bounds on $u$ must be verified for the PDE at
hand and estimates involve embedding and Poincar\'{e}--Friedrichs-type 
constants, where often only very conservative estimates are available
if at all. A common means to improve such bounds and tighten the
McCormick relaxation is optimization-based bound-tightening (OBBT)  
\cite{quesada1993global,quesada1995global} but this may be deemed to be 
computationally intractable due to, at first infinitely many,
additional mixed control-state constraints with infinitely many bounds to
be tightened. This problem generally persists after discretization
due to the high number of bounds to be tightened at, e.g., every node of
a nodal basis.

For this reason, we reduce the number of constraints by performing a local averaging of
the bilinear term and thus also the McCormick inequalities. The resulting 
approximation is a system of finitely many linear inequalities whose
number can be controlled independently thus chosen much coarser
than of the underlying mesh, on which the PDE is solved. This allows to deal with much fewer bounds to tighten and a much lower
computational load. The second means to further reduce the computational
load is to also make a coarser ansatz for the control
discretization, thereby reducing the number of variables, the
optimization has to deal with in a reduced formulation.
Clearly, these discretizations involve errors that need to be quantified
and subtracted if one wants to use these approximate McCormick relaxations
for dual bounds in global optimization algorithms.
In \cite{manns}, this procedure was analyzed for a general class
of PDE control problems and showed that the derived problem provides
an approximate dual bound on the original problem up to an a priori
error estimate under a set of assumptions on the PDE and the involved
discretizations. They provide a simple one-dimensional example that 
satisfies these assumptions by using embeddings that are only available
in 1D. We show that these assumptions and a priori error estimates can
be actually be verified for the multi-dimensional model problems
\eqref{eq : ocpe} and \eqref{eq : ocpc} using two different
proof strategies so that our verification shall serve as a guidance
for future research. We take additional care to include
a suitable discretization of the $\TV$-term in our analysis, thereby settling an open problem from \cite{schiemann2025discretization}.

While the verification of the assumptions of our model problems is our focus, we also provide computational results that build on the
two-dimensional teasing example from \cite{manns}, which is structurally
similar to \eqref{eq : ocpe}. In \cite{manns}, it was demonstrated
that the OBBT procedure consumes very high compute times already for
a relatively coarse local averaging grid. We show how different warm start techniques and a
deliberate order of grid cells can
accelerate the OBBT procedure substantially.

The remainder of this article is structured as follows. We first briefly introduce our notation below.
Then, in \cref{sec : mcc_theory}, we introduce the McCormick relaxations and present how and under which assumptions we approximate the model problems. Since the necessary assumptions are nontrivial for the problems considered, we extensively verify them in \cref{sec : verification_of_ass}. In the course of this, we first verify the assumptions concerning the Laplace-type PDE governing \eqref{eq : ocpe} and Helmholtz equation constraining \eqref{eq : ocpc}
in \cref{subsec : verification_of_ass_ocpe} and \cref{subsec : verification_of_ass_helmholtz} respectively. Then, in \cref{subsec : verification_of_ass_tv}, we address the approximation of the total variation for our domain. We describe our computational experiments
and present and discuss their results in \cref{sec : comp_results}. We conclude the findings of this paper in \cref{sec : concluson}.
\section{Notation and Constants}\label{sec:notation}
Our estimations will involve constants 
which depend on the parameters of the initial problem. We will introduce the constants with respect to these parameters
and then abbreviate them in the remainder omitting the initial problem parameters
in order to avoid notational bloat; e.g., we will abbreviate $c_p(\Omega)$, 
denoting the squared Poincaré constant, by $c_p$.
We list the constants and operators that are frequently used in the remainder.

\paragraph{Operators.} We denote by $\meas$ the Lebesgue measure.

We introduce notation for the projection operator to piecewise constant functions 
on partitions of $\Omega$, which will perform the local averaging of the bilinear 
term.Let ${\calQ}_{h} = \{Q_h^1, \ldots, Q_{h}^{N_h}\}$ be a disjoint partition of 
$\Omega$ for $N_h \in \N$ with mesh size $h$. Then, we denote the projection
from the space  of integrable functions to the piecewise constant functions 
defined on $\{{\calQ}_{h}\}$ by $P_{h}$, which reads
\begin{equation*}
    P_{h} g = \sum_{Q_h^i \in {\calQ}_{h}} \1_{Q_h^i} \underbrace{\frac{1}{\meas Q_h^i} \int_{Q_h^i} g}_{\eqqcolon (P_h g)_i} \text{ for } g \in L^1(\Omega).
\end{equation*}
\paragraph{Vector spaces.} 
For $\mathbb{K} \in \{\R,\C\}$,
$L^p(\Omega, \mathbb{K}^n)$ denotes the Lebesgue space of $p$-integrable functions for $1 \leq p <\infty$; 
in case $n=1$ we simplify this notation to $L^p(\Omega, \K)$.
$L^\infty(\Omega, \mathbb{K})$ denotes the space of essentially bounded functions. We will use $H^1(\Omega, \mathbb{K})$ for the Sobolev space of 
$L^2(\Omega, \mathbb{K})$-functions having $L^2(\Omega,\mathbb{K})$-integrable partial derivatives endowed with the norm 
\[
    \| y \|_{H^1(\Omega, \mathbb{K})} \coloneqq \sqrt{\|y\|^2_{L^2 (\Omega, \mathbb{K})} + \|\nabla y\|^2_{L^2 (\Omega, \mathbb{K}^n)}},
\]
which is equivalent to $\sqrt{\|y\|^2_{L^2 (\partial \Omega, \K)}
+
\|\nabla y\|^2_{L^2 (\Omega,\mathbb{K}^n)}}$ for some $\lambda > 0 $ such that
\begin{equation*}
\sqrt{\lambda}
\|y\|_{H^1(\Omega,\mathbb{K})} \le \sqrt{\|y\|^2_{L^2(\partial \Omega, \mathbb{K})} + \|\nabla y\|^2_{L^2(\Omega,\mathbb{K}^n)}},
\end{equation*}
where $\tr : H^1(\Omega,\mathbb{K})
\to H^{1/2}(\partial\Omega)$
is the trace operator.
We observe that $H^1(\Omega,\mathbb{K})$ is a Hilbert space. 

We denote the space of $H^1(\Omega, \mathbb{K})$-functions 
vanishing at the boundary of $\Omega$ by $H_0^1(\Omega,\mathbb{K})$, which
may be equipped with the norm $  \|  \nabla  y \|_{L^2(\Omega, \K^n)} $.

We use $\Hzerodiv(\Omega)$ for the space of $L^2(\Omega, \K^n)$-integrable functions with $L^2(\Omega)$-integrable divergence and vanishing normal trace.  
\paragraph{Order relation between $\C$-valued functions.} We will often require inequalities for real and imaginary part of complex-valued functions $u$, $v \in L^1(\Omega,\C)$. With a slight abuse of notation and in the interest of a clear presentation, we write $u \le v$ a.e.\ to state that both $\Re u \le \Re v$ and $\Im u \le \Im v$ hold a.e., that is, they hold component-wise.

\section{Relaxation with OBBT and local averaging.}\label{sec : mcc_theory}
Leyffer and Manns proposed a three-stage approximation of nonconvex optimal control
problems with bilinear terms in the state equation in \cite{manns}:
\begin{enumerate}
    \item[Stage 1:] Replacing the bilinear term $uw$ by a convex superset described by linear inequalities (so-called McCormick inequalities).
    \item[Stage 2:] Partial approximation of state and control variables in the McCormick inequalities by local averaging.
    \item[Stage 3:] Combined discretization of the control variable and total variation term.
\end{enumerate}
The resulting approximating problem gives an approximate (up to an a priori
estimate) lower bound on \eqref{eq : ocpe}. In addition, the feasible set 
of this problem is tightened towards its convex hull with OBBT (Algorithm 
1 in the same paper).
We now apply the procedure to an abstract PDE system that covers both the constraining PDEs in
\eqref{eq : ocpe} and \eqref{eq : ocpc}. In addition, we generalize the procedure. While one discretization grid is used for both the bilinearity and the control variable in \cite{manns}, we use 
two separate grids, which provides flexibility for balancing accuracy and performance; see also \cite{kaltenbacher2026locally}.
We go through each of the approximation stages to derive the final approximating problem \eqref{eq : mcchh}. 

\paragraph{Stage 1.}
The general idea for deriving the McCormick relaxations is to
(a) replace the bilinear term $uw$ by a new variable $z$  and (b) introduce new linear so-called McCormick inequalities in $z,u,w$. 
These linear inequalities convexify $z = uw$ and thus the feasible set. 

We state a setting that covers both model problems
\eqref{eq : ocpe} and \eqref{eq : ocpc}. To this end, let $\K \in \{\R, \C\}$. Abstracting from the specific PDE, we seek for a weak solution
$u \in H^1(\Omega,\K)$ to
\begin{equation} \label{eq : abstract_pde}
\begin{aligned}
- \div(\tilde{A} \nabla u) + f_1 u + f_2 wu &= f_0 &&\text{ in } \Omega, \\
u &= g_0 &&\text{ on } \Gamma_1, \\
\frac{\partial u}{\partial n} &= g_1(u) &&\text{ on } \Gamma_2,
\end{aligned}
\end{equation}
where $\tilde{A}$ is a coercive operator with bounded entries, $f_0 \in L^2(\Omega)$, $f_1$ and $f_2$ are smooth functions such that ellipticity
estimates can be made on $u$, $g_0 = 0$ corresponds the Dirichlet boundary condition, and $g_1 =  i k_0 u$ corresponds the Robin boundary condition.
We have $\Gamma_1 = \partial\Omega$, $\Gamma_2 = \emptyset$ for
\eqref{eq : ocpe} and $\Gamma_1 = \emptyset$, $\Gamma_2 = \partial\Omega$
for \eqref{eq : ocpc}. 

We follow Example 1 from \cite{McCormick} to derive the McCormick relaxations of $uw \in L^2(\Omega,\K)$
under uniform bounds $u_\ell \le u \le u_u$
with $u_\ell$, $u_u \in L^2(\Omega,\K)$; see \cref{ass : pdes} below. We
will argue the existence of $L^\infty$-bounds at a later point.

Then, the convex lower bound (see \cite{McCormick} and  \cref{fig:conv_under} for a visualization) and the concave upper bound \cite{al1983jointly} of the function $uw$ are a.e.\ given by 
\begin{equation} \label{eq : mcc_ineq}
\begin{aligned}
wu &\leq \max\{w_u u + u_u w- u_u w_u, w_\ell u + u_\ell w- u_\ell w_\ell\} ,\\
wu &\ge \min\{w_u u + u_\ell w- u_\ell w_u, w_\ell u + u_u w- u_u w_\ell\}, 
\end{aligned}
\end{equation}
where, if $\K =\C$, we define the operators $\max$ and $\min$ component-wise over the real and imaginary part
 such that
 $\max\{a,b\} \coloneqq  \max\{\Re a,\Re b\} + i \max\{\Im a,\Im b\} $ and 
$\min\{a,b\} \coloneqq  \min\{\Re a,\Re b\} + i \min\{\Im a,\Im b\} $ for complex-valued functions $a$, $b$. We also recall that we use the symbols $\le$, $\ge$ component-wise for the real and imaginary part; see also \Cref{sec:notation}.

Due to \eqref{eq : mcc_ineq}, $uw$ can be bounded (component-wise for $\K =\C$) by two linear expressions from below and two from above.

In the real case, this gives 4 pointwise a.e. inequalities
and in the complex case, we have 8 pointwise a.e. inequalities,
which we write as 4 component-wise inequalities in the notation from \cref{sec:notation}.
Replace $uw$ by a new variable $z \in L^2(\Omega, \K)$ that is constrained by the inequalities \eqref{eq : mcc_ineq} yields:
\begin{gather}\label{eq:mcc}
\begin{aligned}
   & - \div(\tilde{A} \nabla u) + f_1 u + f_2 z = f_0 &&\text{ in } \Omega, \\
    &u = g_0 &&\text{ on } \Gamma_1, \\
    &\frac{\partial u}{\partial n} = g_1(u) &&\text{ on } \Gamma_2,\\
&z \ge u_\ell w + u w_\ell - u_\ell w_\ell\enskip\text{a.e.},\\ 
&z \ge u_u w + u w_u - u_u w_u\enskip\text{a.e.},\\
&z \le u_u w + u w_\ell - u_u w_\ell\enskip\text{a.e.},\\
&z \le u_\ell w + u w_u - u_\ell w_u\enskip\text{a.e.},\\
&u_\ell \le u \le u_u\enskip\text{a.e.},\\
&w_\ell \le w \le w_u\enskip\text{a.e.}
\end{aligned}
\end{gather}
The tightness of the McCormick envelope and in turn the McCormick relaxation
depends on the derived bounds $u_\ell$, $u_u$ on $u$. We highlight again that such bounds can be obtained analytically
but involve Sobolev embedding and Poincar\'{e}--Friedrichs-type constants,
which are often not known precisely and may be very large or too conservative in general. 

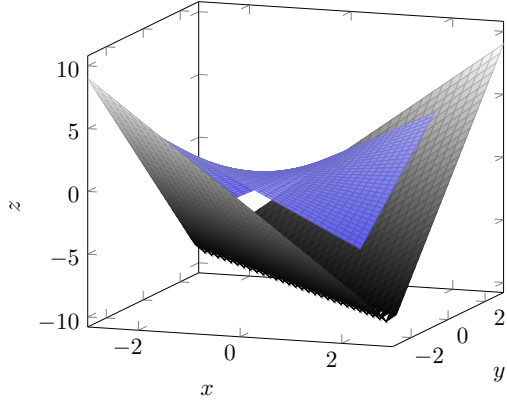
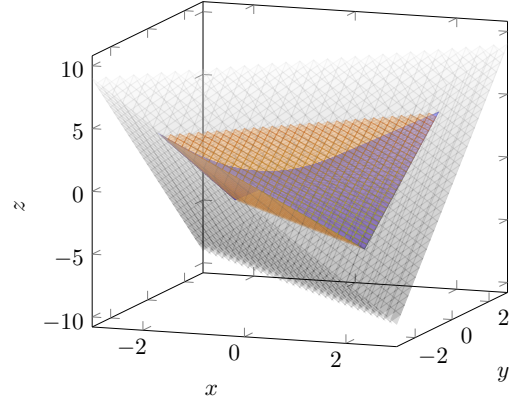
\begin{figure}[t]
\begin{center}
	\begin{subfigure}[t]{0.45\textwidth}
	\begin{tikzpicture}
  \begin{axis}[
    view={20}{12},
    xlabel={$x$}, ylabel={$y$}, zlabel={$z$},
    samples=31, samples y=31,
    z buffer=sort
  ]
    \addplot3[domain=\xl:\xu,y domain=\yl:\yu, surf,  restrict z to domain = -9:9, %colormap/bone,  
    opacity=1, 
    colormap/gray
    %colormap={mygray}{
    %    color(0cm)=(gray!80)
    %    color(1cm)=(gray!15!white)
    %},
    ]
    {\yu*x + y*\xu - \yu*\xu};
    
    \addplot3[domain=-2:2,y domain=-2:2, restrict z to domain = -4:4, surf,      colormap={myblue}{
    color(0cm)=(blue!80)    
    color(1cm)=(blue!15!white)  
    },
    opacity=0.9, 
    ]
    {x*y};
    
    \addplot3[domain=\xl:\xu,y domain=\yl:\yu, restrict z to domain = -9:9, surf, colormap/gray,  
    opacity=1, 
    ]
    {\yl*x + y*\xl - \yl*\xl};

  \end{axis}
\end{tikzpicture} 
		\caption{Convex piecewise linear underestimator (gray) for 
		the term $xy$ (blue) with the choices
        $-x_\ell= x_u = -y_\ell=y_u = 3$.
        }
        \label{fig:conv_under}
	\end{subfigure}
	\hfill
	\begin{subfigure}[t]{0.45\textwidth} 
        \begin{tikzpicture}
  \begin{axis}[
    view={20}{12},
    xlabel={$x$}, ylabel={$y$}, zlabel={$z$},
    samples=31, samples y=31,
    z buffer=sort
  ]
    %\addplot3 [domain=\xl:\xu,y domain=\yl:\yu, surf,  fill =blue, opacity=0.8, 
    %]
    %{min(\yu*x + y*\xl - \yu*\xl,\yl*x + y*\xu - \yl*\xu)};

    %\addplot3 [domain=\xl:\xu,y domain=\yl:\yu, surf, restrict z to domain = -6.25:6.25, fill =orange!70, opacity=1, 
    %]
   % {\yu*x + y*\xl - \yu*\xl};

    %%hitere blaue
    \addplot3 [domain=\xl:\xu,y domain=\yl:\yu, surf, restrict z to domain = -9:9, colormap/gray,  opacity=\greynocapity, 
    ]
    {\yu*x + y*\xl - \yu*\xl};

    %%hitere grüne
    \addplot3[domain=\xlo:\xuo,y domain=\ylo:\yuo, surf, restrict z to domain = \xloylu:\xluylu,%-6.25:6.25, 
    colormap={myorange!80}{
    color(0cm)=(orange)    
    color(1cm)=(orange!15!white)  
    },  opacity=\orangeocapity, 
    ]
    {\yuo*x + y*\xlo - \yuo*\xlo};
    
    \addplot3[domain=\xl:\xu,y domain=\yl:\yu, surf, restrict z to domain = -9:9, colormap/gray,  opacity=\greynocapity, 
    ]
    {\yu*x + y*\xu - \yu*\xu};

    \addplot3[domain=\xlo:\xuo,y domain=\ylo:\yuo, surf, restrict z to domain = \xloylu:\xluylu,%-6.25:6.25, 
    colormap={myorange!80}{
    color(0cm)=(orange)    
    color(1cm)=(orange!15!white)  
    },  opacity=\orangeocapity, 
    ]
    {\yuo*x + y*\xuo - \yuo*\xuo};
    
    \addplot3[domain=-2:2,y domain=-2:2, restrict z to domain = -4:4, surf,  colormap={myblue}{
    color(0cm)=(blue)    
    color(1cm)=(blue!15!white)  
    },  opacity=1.0, 
    ]
    {x*y};
    
    \addplot3[domain=\xlo:\xuo,y domain=\ylo:\yuo, restrict z to domain = \xloylu:\xluylu,%-6.25:6.25, 
    surf,  colormap={myorange!80}{
    color(0cm)=(orange)    
    color(1cm)=(orange!15!white)  
    },  opacity=\orangeocapity, 
    ]
    {\ylo*x + y*\xlo - \ylo*\xlo};

    \addplot3[domain=\xl:\xu,y domain=\yl:\yu, restrict z to domain = -9:9, surf,  colormap/gray,  opacity=\greynocapity, 
    ]
    {\yl*x + y*\xl - \yl*\xl};

    %orange vorne
    \addplot3[domain=\xlo:\xuo,y domain=\ylo:\yuo, restrict z to domain = \xloylu:\xluylu,%-6.25:6.25, 
    surf,  fill =orange!50,  opacity=\orangeocapity, 
    ]
    {\ylo*x + y*\xuo - \ylo*\xuo};

    %blue vorne
    \addplot3[domain=\xl:\xu,y domain=\yl:\yu, restrict z to domain = -9:9, surf,  colormap/gray,  opacity=\greynocapity, 
    ]
    {\yl*x + y*\xu - \yl*\xu};

  \end{axis}
\end{tikzpicture} 
    \caption{McCormick envelope for the term $xy$
    	(in blue)
    before (gray) and after
    OBBT (orange) with initial bounds
    $-x_\ell= x_u = -y_\ell=y_u = 3$
    and tight bounds $\bar{x}_\ell= \bar{x}_u = \bar{y}_\ell=\bar{y}_u = 2$.}
    \label{fig:mcc_obbt}
	\end{subfigure}
\end{center}
\caption{McCormick envelope for a real finite-dimensional biliniarity $z = xy$  in $[-2,2]^2\times [-4,4]$ (blue)
for 
    $x_\ell \leq x \leq x_u, y_\ell \leq y \leq y_u$.}
\label{fig:mcc}
\end{figure}
The disadvantage of the family of inequalities in \eqref{eq:mcc}
is that they add pointwise linear mixed state-control constraints to the problem, which results in many additional inequalities after discritization.
This becomes more aggravated if one decides to improve the approximation 
by tightening the bounds on $u$ (see Remarks 2.3, 2.6 in \cite{manns}).
For this reason, the next stage performs a partial discretization of
the bilinear term $uw$ on an---ideally relatively coarse---grid before
the derivation of the McCormick inequalities. 

\begin{remark}
    Here, one has a choice for the McCormick inequalities, specifically one can choose between the following three options
    \begin{enumerate}[label=(O\arabic*)]
        \item\label{opt : 1} $z \coloneqq uw$,
        \item\label{opt : 2}  $z \coloneqq f_2 uw$,
        \item\label{opt : 3}  $z \coloneqq f_1 u + f_2 uw$.
    \end{enumerate}
	Since the arguments are completely analogous and all McCormick inequalities that are valid for \ref{opt : 2} and \ref{opt : 3} are also valid for \ref{opt : 1}, we consider only \ref{opt : 1} in this work.
\end{remark}

\paragraph{Stage 2.}
We define the abstract PDE with locally averaged bilinear term as
\begin{gather} \label{eq : ocp_proj}
\begin{aligned} &- \div(\tilde{A} \nabla u) + f_1 u + f_2 (P_h u)(P_h w) = f_0 &&\text{ in } \Omega\\
    &u = g_0 &&\text{ on } \Gamma_1, \\
    &\frac{\partial u}{\partial n} = g_1(u) &&\text{ on } \Gamma_2.
\end{aligned} 
\end{gather}
The operator $P_h$ maps to a constant value
per grid cell of the partition of $\Omega$ so that the locally averaged bilinear term reads
\begin{gather}\label{eq : ocp_proj_nh}
\begin{aligned}
    (P_h u)(P_h w)
   &= \sum_{i \in \{1,\ldots,N_h\}} 
   \Big(\frac{1}{\meas Q_h^i} \int_{Q_h^i} u\Big)
   \Big(\frac{1}{\meas Q_h^i} \int_{Q_h^i} w\Big)   
   \1_{Q_h^{i}},
\end{aligned}
\end{gather}
where $\calQ_{h} = \{Q^1_h,\ldots,Q^{N_h}_h\}$
with mesh size $h$ for some $N_h \in \N$ is a partition
of $\Omega$.

Applying the McCormick relaxation approach from Stage 1
on each summand in \eqref{eq : ocp_proj_nh} yields the following approximation of the linear
 inequality system \eqref{eq:mcc}:
\begin{gather}\label{eq:mcch}
\begin{aligned}
    & - \div(\tilde{A} \nabla u) + f_1 u + f_2 z = f_0 &&\text{ in } \Omega, \\
    &u = g_0 &&\text{ on } \Gamma_1, \\
    &\frac{\partial u}{\partial n} = g_1(u) &&\text{ on } \Gamma_2,\\
    &z = \sum_{i=1}^{N_h} z_i \1_{Q^i_h},\\
&z_i \ge u_\ell^i (P_h w) + (P_h u) w_\ell^i - u_\ell^i w_\ell^i
\enskip && \text{ on }Q^i_h\text{ for all } i \in \{1,\ldots,N_h\},\\ 
&z_i \ge u_u^i (P_h w) + (P_h u) w_u^i - u_u^i w_u^i
\enskip && \text{ on }Q^i_h\text{ for all } i \in \{1,\ldots,N_h\},\\ 
&z_i \le u_u^i (P_h w) + (P_h u) w_\ell^i - u_u^i w_\ell^i
\enskip && \text{ on }Q^i_h\text{ for all } i \in \{1,\ldots,N_h\},\\ 
&z_i \le u_\ell^i (P_h w) + (P_h u) w_u^i - u_\ell^i w_u^i
\enskip && \text{ on }Q^i_h\text{ for all } i \in \{1,\ldots,N_h\},\\ 
& u_\ell^i \le P_h u \le u_u^i
\enskip && \text{ on }Q^i_h\text{ for all } i \in \{1,\ldots,N_h\},\\
&w_\ell \le w \le w_u\enskip\text{a.e.}\\
\end{aligned}
\end{gather}
The McCormick inequalities are defined per cell $Q_h^i$ and thus their 
number remains constant during a subsequent discretization of the PDE.
We note that the bound constraints on $w$ remain valid for $P_h w$,
but we now require bounds on $P_h u$ per cell $Q_h^i$.

While it is already computationally meaningful to perform OBBT
on $P_h u$ at this point, we proceed with an approximation of the
control variable $w$, which may be used to balance accuracy and 
approximation quality \cite{kaltenbacher2026locally}.

\paragraph{Stage 3.} We discretize $w$ as piecewise constant functions
on a further partition $\tilde{\calQ}_{\tau} = \{ \tilde{Q}_\tau^1,\ldots,\tilde{Q}_\tau^{N_\tau}\}$
 with mesh size $\tau$ for some $N_\tau \in \N$, that is, we choose
 the ansatz
\[ w = \sum_{j=1}^{N_\tau} \1_{\tilde{Q}_\tau^j} w_j,
\quad w_j \in [w_\ell, w_u] \text{ for } j \in \{1,\ldots,N_\tau\}.
\]
We assume that $\tilde{\calQ}_{\tau}$ is embedded into $\calQ_h$.
Specifically, each grid cell $Q_h^i$ is the finite union of grid
cells $\tilde{Q}_\tau^j$. The resulting approximation of \eqref{eq:mcch}
reads
\begin{gather}\label{eq:pdehtau}
\begin{aligned}
    & - \div(\tilde{A} \nabla u) + f_1 u + f_2 z = f_0 &&\text{ in } \Omega, \\
    &u = g_0 &&\text{ on } \Gamma_1, \\
    &\frac{\partial u}{\partial n} = g_1(u) &&\text{ on } \Gamma_2,\\
& w = \sum_{j=1}^{N_\tau} w_j \1_{\tilde{Q}^j_h}, z = \sum_{i=1}^{N_h} z_i \1_{Q^i_h},\\
&z_i \ge u_\ell^i (P_h w) + (P_h u) w_\ell^i - u_\ell^i w_\ell^i
\enskip && \text{ on }Q^i_h\text{ for all } i \in \{1,\ldots,N_h\},\\ 
&z_i \ge u_u^i (P_h w) + (P_h u) w_u^i - u_u^i w_u^i
\enskip && \text{ on }Q^i_h\text{ for all } i \in \{1,\ldots,N_h\},\\ 
&z_i \le u_u^i (P_h w) + (P_h u) w_\ell^i - u_u^i w_\ell^i
\enskip && \text{ on }Q^i_h\text{ for all } i \in \{1,\ldots,N_h\},\\ 
&z_i \le u_\ell^i (P_h w) + (P_h u) w_u^i - u_\ell^i w_u^i
\enskip && \text{ on }Q^i_h\text{ for all } i \in \{1,\ldots,N_h\},\\ 
& u_\ell^i \le P_h u \le u_u^i &&  \text{ on }Q^i_h\text{ for all } i \in \{1,\ldots,N_h\},\\
&w_\ell \le w \le w_u\enskip\text{a.e.},
\end{aligned}
\end{gather}
where we observe that $P_h w = P_\tau w = w_i$ on $Q_h^i$ holds for all
$i \in \{1,\ldots,N_h\}$ if the partitions $\{Q_h^i\}$ and $\{\tilde{Q}_\tau^j\}$
coincide, that is, if the same discretization grids are used for the
bilinearity and $w$. If this is not the case, $P_h w$ is a finite
sum on $Q_h^i$.

While this ansatz implies $w = P_\tau w$, we now have to take extra care of
the total variation term since, for a given $w \in L^1(\Omega)$,
$\TV(P_\tau w) \to \TV(w)$ does not necessarily hold as $\tau \to 0$;
see, e.g., Example 2.10 in \cite{schiemann2025discretization}. Consequently, we require a
discretization, which we denote by $\TV_{\tau}$, of the $\TV$-functional
that does not have this problem; see \cite{caillaud2023error,chambolle2017accelerated}
or, in the case of additional discreteness or integer restrictions on $w$,
\cite{schiemann2025discretization} for such discretizations.

The resulting optimization problem reads
\begin{gather} \label{eq : mcchh}
\begin{aligned}
\min_{\substack{u, w_1,\ldots,w_{N_\tau} \\
		z_1,\ldots,z_{N_h}}}\ & j(u, w) + \alpha \text{TV}_\tau(w)\\
\text{s.t. } &u,w,z \text{ satisfy } \eqref{eq:pdehtau} \\
&u \in H^1(\Omega, \K),\; 
z_1,\ldots,z_{N_h} \in \K,\;
w_1,\ldots,w_{N_\tau} \in \R.
\end{aligned}\tag{McC$_{h\tau}$}
\end{gather}
\begin{remark}
The assumption that the grid with mesh size $\tau$ is embedded into the grid with mesh size $h$ is not necessary for most of our arguments.
but we use it here for simplicity of the argument and to provide a more intuitive understanding. This assumption has also been made to derive the approximation
results in \cite{kaltenbacher2026locally}.
\end{remark}

\paragraph{OBBT and Approximate Lower Bounds.}
We can apply
the OBBT Algorithm 1 from \cite{manns} on \eqref{eq : mcchh} to reduce its feasible set.
The idea of the algorithm is to solve a sequence of linear problems
\begin{equation}\label{eq : obbt_problems}\tag{OBBT$_{\ell / u }(i)$} 
    \tilde{u}_\ell^i / \tilde{u}_u^i \leftarrow \min / \max \big\{(P_h u_h)|_{Q^i_h} : (u_h, w_\tau, z_h)\text{ is feasible for \eqref{eq : mcchh}}\big\}
\end{equation} over all partitions $Q^i_h$. 
As in the approximation stage 1, in the case $\K =\C$,
 we define the operators $\max$ and $\min$ component-wise over 
 the real and imaginary part
such that
$\max\{a,b\} = \max\{\Re a,\Re b\} + i \max\{\Im a,\Im b\} $ and 
$\min\{a,b\} =\min\{\Re a,\Re b\} + i \min\{\Im a,\Im b\} $ for complex-valued functions $a,b$. 
In such a way we actually need to (repeatedly) solve $4N_h$ problems in the complex case.
Subsequently, the bounds $u_l^i, u_u^i$ on $(P_h u_h)$ are being updated
 with the obtained values $\tilde{u}_\ell, \tilde{u}_u$ and the bound-tightening procedure is being repeated. 
 The values of $\tilde{u}_\ell, \tilde{u}_u$ are well defined;
 see Lemma 2.17 in \cite{manns}.
Let us denote the new feasible set obtained after the $n$-th round of bound-tightening by $\mathcal{F}^n$. 
Then, the sequence of induced approximate McCormick relaxations
\begin{gather}\label{eq:mcchhn}
 \underset{\substack{u, w_1,\ldots,w_{N_\tau}  \\
		z_1,\ldots,z_{N_h}}}{\min}\ j(u, w) + \alpha\TVt(w)
\quad\text{\emph{s.t.}}\quad
(u, w_1,\ldots,w_{N_\tau} , z_1,\ldots,z_{N_h})  \in \calF^n,
\tag{McC$_{h\tau}^n$}
\end{gather}
provides approximate lower bounds on \eqref{eq : ocpe} that satisfy
\begin{gather}\label{eq:mcchh_approx_lb_ocp}
m_{\eqref{eq : ocpe}} \ge m_{\eqref{eq:mcchhn}} - L_u^n \| \bar{u} -  \bar{u}_h\|_{L^2(\Omega, \K)}
- L_w \sqrt{(w_u - w_\ell)}n^\frac{1}{4}
  \TV(\bar{w})^\frac{1}{2} \tau^\frac{1}{2},
\end{gather}
where $m_{\eqref{eq : ocpe}}$ is the infimum of \eqref{eq : ocpe},
$(\bar{u},\bar{w})$ is a minimizer of \eqref{eq : ocpe},  
$L_u^n$ and $L_w$ are the Lipschitz constants of $j$ with respect to its first and second argument on $\mathcal{F}^n \cup \argmin \eqref{eq : ocpe}$ and $W_{ad}$ resp.\ \cite[Theorem 2.11]{manns}
(note that $U$ 
therein
 can be replaced by $L^2(\Omega, \K)$ without changing the proof).
The error term can be estimated a priori; see \eqref{eq : apriori}.

\paragraph{Assumptions.}
In order to apply OBBT to \eqref{eq : mcchh} and to estimate the approximation error, the original problem \eqref{eq : ocpe} 
and its approximation \eqref{eq : mcchh} have to satisfy several assumptions, which
we split into three blocks. The first block concerns the partitions $\calQ_h$.

\begin{assumption}[Assumptions on grid cells, Assumption 2.7 from \cite{manns}] \label{ass : grid_cells} 
We assume the following. 
    \begin{enumerate}
     \item $\{\calQ_h\}_h$ is a sequence of partitions of the domain $\Omega$ with mesh sizes $h$, 
     \item \label{itm:bounded_eccentricity} $\{\calQ_h\}_h$ satisfies a bounded eccentricity condition, that is, there exists $C > 0$ such that for each $Q_h^i$ there is a closed ball $B_h^i$ such that
     $Q_h^i \subset B_h^i$ and $\meas B_h^i \le C \meas Q_h^i$.
     \end{enumerate}
\end{assumption}
The next block of assumptions ensures the existence, boundedness, and approximability of the solutions of the original and locally averaged PDEs. 
\begin{assumption}[Assumptions on PDEs under assumptions on grid cells]\label{ass : pdes} 
In addition to \cref{ass : pdes}, we assume the following.
    \begin{enumerate}
        \item\label{ass : wu_bounds_exist} There are bounds $u_\ell$, $u_u \in L^\infty(\Omega, \K)$
        such that $u_\ell \le u \le u_u$ a.e.\ holds for solutions $u$ to \eqref{eq : abstract_pde}
        uniformly over all $w \in \wad$. 
        \item\label{ass : phwphu_pdesolexists} The PDE and the one with locally averaged bilinear term \eqref{eq : abstract_pde} and \eqref{eq : ocp_proj} 
        admit unique solutions for all $w \in \wad$. 
        \item\label{ass : richness_of_feasible_set} For fixed $h$ and all $Q_h^i \in \calQ_h$,
        the bounds $u_\ell^i$, $u_u^i \in \K$
        satisfy
        $u_\ell^i \le (P_h u_h)|_{Q_h^i} \le u_u^{i}$ for all
        $u_h$
        that solve \eqref{eq : ocp_proj} for $w \in \wad$.  
        \item\label{itm:ubnd_limits} $u_{\ell,h} \coloneqq \sum_{i=1}^{N_h} u_\ell^i \1_{Q_h^i}$ and
      $u_{u,h} \coloneqq \sum_{i=1}^{N_h} u_u^i \1_{Q_h^i}$ satisfies
      $u_{\ell,h} \to u_\ell$ in $L^2(\Omega, \K)$ and $u_{u,h} \to u_u$ in $L^2(\Omega, \K)$.
        \item\label{itm:uuh_approx} Let $u$ solve \eqref{eq : abstract_pde} for some $w \in \wad$.
        Then $u_h \to u$ in $H^1(\Omega, \K)$ holds for the solutions $u_h$ to 
        \eqref{eq : ocp_proj} for $w$ for $h\searrow 0$. 
    \end{enumerate}
\end{assumption}
The last block of assumptions concerns the correct approximation of the total variation. We give a more general assumption than Assumption 2.12 in \cite{manns}
that allows us to cover all of the three aforementioned discretizations
of $\TV$ from \cite{caillaud2023error,schiemann2025discretization}.
\begin{assumption}[Assumptions on the approximation of $\TV$] \label{ass : bv} 
We assume that there exists a sequence of approximations $(\TV_\tau)_\tau$ of $\TV$ in the following sense:
\begin{enumerate}
    \item \label{itm:consistency} For all $\tau$, $\TV_\tau : \wad \rightarrow [0, \infty]$ is lower semicontinuous with respect to
    $L^1$-convergence,  
    $\TV_\tau \circ P_\tau \leq \TV$, and $\TV_\tau \circ P_\tau \leq \TV_\tau$,
    \item \label{itm:TVh_compactness} $\sup_{\tau\searrow 0} \TV_\tau(w_\tau) \le C$ 
    and $\sup_{\tau\searrow 0} \|w_\tau\|_{L^1(\Omega)} \le C$ for some $C > 0$ with $w_\tau = P_{\tau} w_\tau$ imply
    a subsequence $\{w_{\tau_k}\}_{k}$ and $w \in \BV(\Omega)$ such that
    $w_{\tau_k} \to w$ in $L^1(\Omega)$,
    \item \label{itm:TVh_gamma} $\TV_\tau \circ P_{\tau}$ $\Gamma$-converges to $\TV$ on $\wad$ for $\tau \searrow 0$
    with respect to $L^1$-convergence.
\end{enumerate}
\end{assumption}
\begin{remark}
If $\wad$ is of the form $\wad = \{ w \in L^1(\Omega) :
w(x) \in \{\omega_1,\ldots,\omega_M\} \text{ a.e.}\}$
for finitely many $\omega_1$, $\ldots$, $\omega_M \in \R$,
one may adapt the discretization from 
\cite{schiemann2025discretization}.
Then \cref{ass : bv} can still be satisfied
if one defines $\TV_\tau \coloneqq 
\max\{\frac{1}{c}\TV, \TV_{\sigma} \}$
with $\TV_\sigma$ defined in \cite{schiemann2025discretization}
(use $h = \sigma$ therein) when
$\tau/\sigma \searrow 0$ as $\sigma \searrow 0$
and $\tau \searrow 0$;
see the results in \S3 in \cite{schiemann2025discretization}.
\end{remark}
\section{Verification of the assumptions for the model problems} \label{sec : verification_of_ass}
We now verify the remaining required assumptions. We start with the verification of 
\cref{ass : pdes} for \eqref{eq : ocpe}, for which we derive the $L^\infty$-bounds
on the state variable such that we can observe, which constants enter the
this global estimate. For \eqref{eq : ocpc} we build on higher regularity, that is, $H^2$-regularity, from existing literature and use a bootstrapping argument
and do not elaborate on the constants with the same level of detail. While
it is of course also possible to derive them, we note that they become very
cumbersome and we thus leave it to the interested reader.
\subsection{Verification of the assumptions for \texorpdfstring{\eqref{eq : ocpe}}{OCPE}}\label{subsec : verification_of_ass_ocpe}
We recall the constraining elliptic PDE in the weak form
\begin{equation} \label{eq : elliptic_pde} 
\int_\Omega  A \nabla u \cdot \nabla v + \int_\Omega  uwv = \int_\Omega fv \text{ for all } v \in H_0^1(\Omega).
\end{equation}
and observe that it admits a unique solution for all $w \in\wad$.

Specifically, it holds that
$\int_\Omega  A \nabla u \cdot \nabla u + \int_\Omega w u^2 \ge (\beta + \min\{c_p^2 w_\ell,0\})\|\nabla u\|_{L^2(\Omega)}^2$,
where $\beta + \min\{c_p^2 w_\ell,0\} \ge 0$ holds by assumption so that we can apply the Lax--Milgram lemma. 
This also yields the estimation $\| \nabla u \|_{L^2(\Omega)}\le c_1(c_p, w_\ell, \beta)^{-1} \| f\|_{L^2(\Omega)} $ 
for $c_1(c_p, w_\ell, \beta)^{-1} \coloneqq c_p /(\beta + \min\{c_p^2 w_\ell,0\})$.

\paragraph{\Cref{ass : pdes} \ref{ass : wu_bounds_exist}.}
We show that \Cref{ass : pdes} \ref{ass : wu_bounds_exist} is satisfied for
all choices
$u_\ell$, $u_u$
such that
\begin{equation}
u_\ell \le \underline{u}_\ell \coloneqq -c_2(c_p, w_\ell, \beta, \Omega)^{-1}\|f\|_{L^2 (\Omega)} \in \R
\end{equation}
and
\begin{equation}
u_u \ge \overline{u}_u \coloneqq c_2(c_p, w_\ell, \beta, \Omega)^{-1}\|f\|_{L^2 (\Omega)} \in \R,
\end{equation}
where 
\[
c_2(c_p, w_\ell, \beta, \Omega)^{-1} \coloneqq \frac{1}{\min\bigr\{c_p^2 w_\ell, 0\bigr\} + \beta} \cdot \begin{cases} \frac{1}{2} &\text{if }n = 1, \\
8 C_0^{imb}(6)C_0^{imb}(4) (\meas \Omega)^\frac{1}{12}  &\text{if }n \in \{2,3\}.
\end{cases}
\]
The bounds for the one-dimensional case can be obtained directly from the imbedding
$H_0^1(\Omega) \hookrightarrow L^\infty(\Omega)$ and the $H_0^1(\Omega)$-estimate on $u$
due to the Lax--Milgram lemma \cite{manns}.

The multidimensional case requires another approach since the imbedding $H_0^1(\Omega) \hookrightarrow L^\infty(\Omega)$ 
exists only in the one-dimensional case. We follow an argument due to
Stampacchia and Kinderlehrer, specifically \cite[Lemma~B.2]{kinderlehrer} and \cite[Theorem~4.5]{troltzsch2010optimal}.
Because their settings do not cover the case $w_\ell < 0$,
we briefly extend their argument below.

\begin{theorem}\label{thm:ocpe_bound}
Let $-\frac{\beta}{c_p^2} < w_\ell < w_u$. Then, for all $w \in L^\infty(\Omega)$ with $w_\ell \le w \le w_u\enskip\text{a.e.}$, 
%Then, 
the solution $u$ to \eqref{eq : elliptic_pde} satisfies
$u_\ell \le u \le u_u$ for $u_\ell$, $u_u$ above.
\end{theorem}
\begin{proof}
In case $u \equiv  0$, the estimation is trivial. Therefore, we assume $u \not\equiv 0$.

First, we define the auxiliary variables and describe the idea.
Let $k >0$ be fix. We modify the test function from \cite[Lemma B.2]{kinderlehrer} by scaling it with $M \coloneqq \min\bigr\{\frac{c_p w_\ell}{\beta}, 0\bigr\} + 1$:
\begin{equation*}
\zeta_M \coloneqq M\sgn({u}) \max(|{u}|-k,0) = 
    \begin{cases}M({u}-k) & \text{ for } u\geq k, \\ 0 & \text{ for } |{u}| < k, \\ 
    M( {u}+k) & \text{ for } {u}> -k.
    \end{cases} 
\end{equation*}
We note that the scalar $M$ is positive due to the assumption $w_\ell > -\beta / c_p$.
We denote by $E(k)$ the set on which the test function attains a non-zero value:
\begin{equation*}
E(k) \coloneqq \{x \in \Omega : |u(x)| \geq k \}.
\end{equation*}
The derivative of $\zeta_M$ is the $M$-fold derivative of $u$ on the set $E(k)$
and zero elsewhere:
\begin{equation}\label{eq : zeta_der}
  \nabla \zeta_M =  \begin{cases}  M \nabla u & \text{ on } E(k) \\ 0 & \text{ else } |u| < k.\end{cases}.
\end{equation}
Our aim is to show that $E(k)$ is a null set for a sufficiently large $k > 0$. 
We will employ to this end the similarity of $ \nabla \zeta_M$ and $ \nabla {u}$ to show 
\begin{align} \label{eq: bound_zetaM}
\frac{M^2}{(C_0^{imb}(6))^2}(h-k)^2 [\meas E(h)]^{\frac{1}{3}} \leq \|\nabla \zeta_M\|^2_{L^2(\Omega,\R^n)}
 \leq\frac{(C_0^{imb}(4))^2}{ \beta^2} \|f\|_{L^2(\Omega)}^2 [\meas E(k)]^{\frac{1}{2}}
\end{align}
for all positive $h,k$ with $h > k > 0$, where the map $\meas(\cdot)$ and the constants 
$C_0^{imb}(\cdot)$ are defined in \eqref{sec:notation}.
Then, the inequality \eqref{eq: bound_zetaM} can be brought into the form required by
\cite[Lemma B.1]{kinderlehrer}, which implies $\meas E(h) = 0$ for all large enough $h$
and the claimed bounds on $u$.
\paragraph{Estimation of $\|\nabla \zeta_M\|^2_{L^2(\Omega,\R^n)}$ from below.} This estimation is based on the construction of $\zeta_M$ and the Sobolev imbedding theorem (\cite[Theorem 4.5]{troltzsch2010optimal}; choose $c \coloneqq \frac{(C_0^{imb}(6))^2}{M^2}$, $p \coloneqq 6$ therein and the modified test function $\zeta_M$, which causes
the factor $M^2$ in \eqref{eq: bound_zetaM}).

\paragraph{Estimation from $\|\nabla \zeta_M\|^2_{L^2(\Omega,\R^n)}$ above.} First we show the following estimation
\begin{equation*}
\begin{aligned} 
\beta \| \nabla \zeta_M \|^2_{L^2(\Omega,\R^n)} 
\leq \int_{E(k)} f \zeta_M, \label{eq : upper_estim}
\end{aligned} 
\end{equation*}
which corresponds in \cite[(7.8)]{troltzsch2010optimal} but must be obtained differently than in the steps (ii), (iii) therein. Then, the rest of the proof coincides with steps (iv) and (v) in \cite{troltzsch2010optimal} using our modified test function.
To this end, we show 
\begin{equation}\label{eq : fzeta_estim}
\begin{aligned} 
\beta \| \nabla \zeta_M \|^2_{L^2(\Omega,\R^n)} \leq  \int_{E(k)} A\nabla  {u} \cdot \nabla \zeta_M   +   \int_{E(k)} w {u} \zeta_M.
\end{aligned} 
\end{equation}

Because of the ellipticity of the PDE and due to the construction of $\nabla \zeta_M$ (see \eqref{eq : zeta_der}), 
we can estimate $\beta \| \nabla \zeta_M \|^2_{L^2(\Omega,\R^n)}$ as:
\begin{equation*} 
\begin{aligned}
\beta\| \nabla \zeta_M \|^2_{L^2(\Omega,\R^n)} 
\leq \int_{\Omega} A \nabla  \zeta_M \cdot \nabla \zeta_M
= \int_{E(k)} A (M \nabla  {u}) \cdot \nabla \zeta_M . 
\end{aligned}
\end{equation*}
As the constant $M$ attains values depending on the sign of $w_\ell$, we distinguish between the following subcases:

\subparagraph{$w_\ell$ is nonnegative.} We can overestimate the right-hand side by the left-hand side of the PDE \eqref{eq : elliptic_pde} with $w_\ell$ instead of $w$. 
It holds $M = 1$ in the considered case and thus we obtain
\begin{equation*}
\begin{aligned} 
\beta\| \nabla \zeta_M \|^2_{L^2(\Omega,\R^n)}
&\leq  \int_{E(k)} A \nabla  {u} \cdot \nabla \zeta_M   +   \int_{E(k)} w_\ell {u} \zeta_M \\
&\overset{\mathclap{w_\ell \le w \text{ a.e.}}}{\le}\quad \int_{E(k)} A \nabla  {u} \cdot \nabla \zeta_M   +\int_{E(k)} w u \zeta_M.
\end{aligned} 
\end{equation*}

\subparagraph{ $w_\ell$ is negative.} We use the construction of $M$ to obtain the following sequence of estimations
\begin{equation*}
\begin{aligned}
\beta\| \nabla \zeta_M \|^2_{L^2(\Omega,\R^n)}
= &\int_{E(k)} A \nabla  {u} \cdot \nabla \zeta_M +  (M-1) \int_{E(k)} A \nabla  {u} \cdot \nabla \zeta_M  && \scriptstyle{M = M + 1}\\
\leq &\int_{E(k)} A \nabla  {u} \cdot \nabla \zeta_M + \frac{c_p^2 w_\ell}{\beta}\int_{E(k)}  A \nabla  {u} \cdot \nabla \zeta_M && \scriptstyle{M = c_p^2 w_\ell + 1}\\
=& \int_{E(k)} A \nabla  {u} \cdot \nabla \zeta_M  + \frac{c_p^2 w_\ell}{\beta} M \int_{E(k)}  A \nabla  {u} \cdot \nabla u.  && \scriptstyle{ \nabla \zeta_M = M\nabla {u} \text{ on }E(k)}
\end{aligned}
\end{equation*}
Because of the negative sign of $w_\ell$, we need to find a lower bound on the second summand.
For this reason, we use here the elliptic properties of the operator $A$,  Poincaré inequality with a constant $c_p$, and an auxiliary estimation in \cref{thm : aux_estim_M} to transform the
second summand,
\begin{equation*}
\begin{aligned}
\beta\| \nabla \zeta_M \|^2_{L^2(\Omega,\R^n)} 
\leq &\int_{E(k)} A \nabla  {u} \cdot \nabla \zeta_M   +  c_p^2 w_\ell  M\int_{E(k)}  \|\nabla  u\|_{\R^n}^2  &&\scriptstyle{w_\ell < 0 \text{ by assumption}\Rightarrow \atop \text{Elliptic cond. has the reverse ineq.\ sign}}\\
\leq &\int_{E(k)} A \nabla  {u} \cdot \nabla \zeta_M   +  w_\ell M\int_{E(k)}  u^2  &&\scriptstyle{w_\ell < 0 \text{ by assumption}\Rightarrow \atop \text{Poincaré ineq.\ has the reverse ineq.\ sign}}\\
\leq & \int_{E(k)} A \nabla  {u} \cdot \nabla \zeta_M   + w_\ell \int_{E(k)}  {u}\zeta_M &&\scriptstyle{\text{\cref{thm : aux_estim_M}}} \\
\leq &\int_{E(k)} A \nabla  {u} \cdot \nabla \zeta_M   +\int_{E(k)} w u \zeta_M,  &&\scriptstyle{w_\ell \le w \text{ a.e.}}
\end{aligned}
\end{equation*}
yielding \eqref{eq : fzeta_estim}.

Having now established \eqref{eq : fzeta_estim} for both (that is all) cases, we obtain
\[ \beta\| \nabla \zeta_M \|^2_{L^2(\Omega,\R^n)} \overset{\text{PDE }\eqref{eq : elliptic_pde}}{\le} \int_{E(k)} f \zeta_M .
\]

Further we use the estimations from (iv), (v) from the proof of \cite[Theorem 4.5.]{troltzsch2010optimal}. 
Combination of these two steps yields
\begin{align*}
\|\nabla \zeta_M\|^2_{L^2(\Omega,\R^n)} \leq\frac{(C_0^{imb}(4))^2}{\beta^2} \|f\|_{L^2(\Omega)}^2 [\meas E(k)]^{\frac{1}{2}}.
\end{align*}

\paragraph{Application of Lemma B.1\@ from \cite{kinderlehrer}.} We multiply the right- and the left-hand sides of the double-inequality \eqref{eq: bound_zetaM} by an appropriate constant and evaluate them to the power of $3/2$ 
to obtain for all $h > k > 0$
\begin{equation*}
     [\meas E(h)]\leq  \biggr[\frac{(C_0^{imb}(6))^2(C_0^{imb}(4))^2}{\beta^2 M^2}\|f\|_{L^2(\Omega)}^2 \biggr]^3\frac{1}{(h-k)^6}  [\meas E(k)]^\frac{6}{4}.
\end{equation*} 
This estimation corresponds to the first estimation in \cite[p. 363]{troltzsch2010optimal} and allows to apply \cite[Lemma B.1.]{kinderlehrer}.
Observing that $\meas \Omega = 1$, we obtain the desired estimation
 \begin{align*}
    &\meas E(d) = 0 \\
    \text{for }&d = 8\frac{C_0^{imb}(6)C_0^{imb}(4)}{\beta M}\|f\|_{L^2(\Omega)}.
\end{align*}
\end{proof}

\begin{remark}
It is possible to obtain $L^\infty$-bounds $u_\ell$, $u_u$ on $u$
in fewer steps via a bootstrapping approach and using the results from
Theorem 4.5 in \cite{troltzsch2010optimal} directly.
This results in the (weaker) bound:
\begin{equation*}
    \| {u}\|_{L^\infty(\Omega)} \leq \frac{8}{\beta } C_0^{imb}(4) C_0^{imb}(6) \Biggl(1+\max \{-w_\ell,w_u\}  c_p \|f_0\|_{L^2 (\Omega)}  \frac{1}{\min \{c_p^2 w_\ell, 0\} + \beta} \Biggl).
\end{equation*} 
\end{remark}

We recall the locally averaged abstract PDE \eqref{eq : ocp_proj} for \eqref{eq : ocpe}
\begin{equation}\label{eq : ocpe_proj}
    \int_\Omega  A \nabla u \cdot \nabla v + \int_\Omega  (P_h u) (P_h w)v = \int_\Omega fv \text{ for all } v \in H_0^1(\Omega)
\end{equation}
and verify the corresponding assumptions.
\paragraph{\Cref{ass : pdes} \ref{ass : phwphu_pdesolexists}.} 
We observe that $P_h$ is non-expansive with respect to $L^1(\Omega), L^2(\Omega)$, and $L^\infty(\Omega)$ meaning
\begin{equation}\label{eq : nonexp}
    \|P_h u_h\|_{L^t(\Omega)} \leq \| u_h\|_{L^t(\Omega)} \text{ for } t = 1,2, \infty
\end{equation}
for all $h \in \{1, \ldots, N_h\}$ \cite[\S3.3]{manns}. 
Together with the boundedness of $w$, this implies the coercivity of the locally averaged bilinear 
form with the constant $\beta - c_p^2\max\{|w_u|, |w_\ell|\}$.
Consequently, the Lax--Milgram Lemma yields the existence of unique solutions to \eqref{eq : ocpe_proj} for all $w \in \BV(\Omega)$ with $w_\ell \le w \le w_u$ a.e. 
\begin{remark}
We observe that the applications of the Cauchy--Schwarz and Poincaré inequalities
on the left-hand side yield
\[
 \|\nabla {u}_h\|_{L^2(\Omega,\R^n)} \leq  \frac{c_p}{(\beta - \max\{-w_\ell,w_u\}c_p^2)}  \| f \|_{L^2(\Omega,\R^n)} \eqqcolon \frac{1}{c_3(c_p, w_\ell, w_u, \beta)} \| f \|_{L^2(\Omega,\R^n)}.
\]   
\end{remark}

\paragraph{\Cref{ass : pdes} \ref{ass : richness_of_feasible_set}.}
The non-expansiveness of $P_h$ also yields that it is sufficient to find the bounds on $u_h$ instead of $P_h u_h$. To this end, we use a bootstrapping argument together with \cite[Theorem 4.5]{troltzsch2010optimal} and obtain 
\begin{equation*}
\begin{aligned} 
\|{u_h}\|_{L^\infty (\Omega)} &\leq \frac{8 C_0^{imb}(4) C_0^{imb}(6) }{\beta }  \|f- (P_{h} w) (P_{h} {u_h}) \|_{L^2(\Omega)}.
\end{aligned}
\end{equation*}
Further application of the Hölder and Poincaré inequalities together with the $H_0^1(\Omega)$-bound on $u_h$ and the $L^\infty(\Omega)$-bounds on $w$ (which also hold for $P_h w$ because of the non-expansiveness of $P_h$) yield
\begin{equation}\label{eq:ocpe_locally_averaged_bound}
\begin{aligned} 
\|{u_h}\|_{L^\infty (\Omega)} &\leq \underbrace{\underbrace{\frac{8 C_0^{imb}(4) C_0^{imb}(6) }{\beta }  \biggr(1+\max \{-w_\ell,w_u\}  \frac{c_p}{\beta - \max\{-w_\ell, w_u\}c_p^2}\biggr)}_{\eqqcolon 1/c_4(c_p, w_\ell, w_u, \beta)}\|f\|_{L^2(\Omega)}}_{\eqqcolon u_h^*}.
\end{aligned}
\end{equation}

\paragraph{\Cref{ass : pdes} \ref{itm:ubnd_limits}.}
Since $P_h$ is non-expansive, $\|P_hu_h\|_{L^\infty} \le \|u_h\|_{L^\infty}$
holds.
Consequently, we can choose 
\[ u_{\ell,h} = u_\ell \le \min\{\underline{u}_\ell, -u_h^*\}
\quad\text{and}\quad
u_{u,h} = u_u \ge \min\{\overline{u}_u, u_h^*\}
\]
to certify the assumption while still complying with
\Cref{ass : pdes} \ref{ass : wu_bounds_exist}, \ref{ass : richness_of_feasible_set}. 

\paragraph{\Cref{ass : pdes} \ref{itm:uuh_approx}.} First we a priori estimate the PDE approximation
error induced by local-averaging as
\begin{equation}\label{eq : apriori}
    \|{u} - {u}_h\|_{L^2(\Omega)} \leq C_{3/2}^a(w) d_h^{3/2}
            + C_{3/2}^b d_h^2
\end{equation}
for some positive constants $C_{3/2}^a(w) \coloneqq C_{3/2}^a(n, w_\ell,w_u,f,w,\Omega)$, $C_{3/2}^b \coloneqq C_{3/2}^b(n, w_\ell,w_u,f,\Omega)$ and the diameter $d_h$ of the grid cells. Then, we show that this yields
the following estimate
\begin{equation}\label{eq:ocpe_apriori_gradl2}
\|\nabla(u - u_h)\|_{\Ln}^2 \leq \beta^{-1} \max \{ -w_\ell,w_u\} (c_1^{-1} +c_3^{-1}) \|f\|_{L^2(\Omega)} \bigr(C_{3/2}^a(w) d_h^{3/2}
+ C_{3/2}^b d_h^2\bigr).
\end{equation}
\begin{theorem} \label{thm: apriori}
Let $u$ solve \eqref{eq : elliptic_pde} and $u_h$ solve \eqref{eq : ocpe_proj} for the same $w \in \BV(\Omega)$ with $w_\ell \leq w \leq w_u$ a.e.
Then the  apriori estimates \eqref{eq : apriori}
and \eqref{eq:ocpe_apriori_gradl2} hold.
\end{theorem}
\begin{proof}
    We follow the proof from \cite[Lemma 3.1]{manns}. As the proof covers only the case where $\Omega = (0,1)$ and $A$ is an identity matrix, 
    we comment on the modifications. We use \cite[Theorem 3.1.4]{ciarlet2002finite},
    specifically the estimations 
    \begin{equation}\label{eq : estim_ph}
        \begin{aligned}
            \|g - P_{h} g \|_{L^2(\Omega,\mathbb{K})} &\leq \alpha(\Omega) d_h \|\nabla g\|_{L^2(\Omega,\mathbb{K})} &&\text{ for all } g \in H^1(\Omega,\mathbb{K})
            \text{ for } \mathbb{K} \in \{\R,\C\}, \\
            \| w - P_{h} w \|_{L^1(\Omega)} &\leq \alpha(\Omega) d_h \TV(w) &&\text{ for all } w\in L^1(\Omega)
        \end{aligned}
    \end{equation}
    hold for some $\alpha(\Omega) > 0$. 
    Here, the case $\mathbb{K} = \C$ is a direct corollary of the case $\mathbb{K} = \R$ in the first estimate. The second estimate
    follows from the $W^{1,1}$-estimate in combination with the strict approximability of $w \in L^1(\Omega)$ by smooth functions of
    bounded variation \cite[\S3.9]{ambrosio2000}. Then the argument form \cite[Lemma 3.1]{manns} yields
    \begin{equation*}
    \begin{aligned}
    \|{u} - {u}_h\|_{L^2(\Omega)} &\leq 
     \bigl( \frac{1}{c_2 c_3} + \frac{1}{c_2 c_4}\bigl) \lambda (\Omega)^{\frac{3}{2}} d_h^{\frac{3}{2}}  | w_u - w_\ell |^{\frac{1}{2}} \TV(w)^{\frac{1}{2}} \| f\|_{L^2(\Omega)} \\
    &+\frac{\lambda (\Omega)^2 d_h^2 }{c_2 c_3} \max \{-w_\ell,w_u\}  \| f\|_{L^2(\Omega)},
    \end{aligned}
    \end{equation*}
    which is \eqref{eq : apriori}.
    
    To obtain \eqref{eq:ocpe_apriori_gradl2},
    we test \eqref{eq : elliptic_pde} and
    \eqref{eq : ocpe_proj} with $v = u - u_h$ and use
    coercivity, Poincar\'{e} inequality,
    nonexpansiveness of $P_h$, and \eqref{eq : apriori}:
    \begin{equation*}
    \begin{aligned}
    \beta\|\nabla (u-u_h)\|_{L^2(\Omega, \R^n)}^2 &\le \int_\Omega  A \nabla (u-u_h) \cdot \nabla(u-u_h) \\
    &= 
    \int_\Omega (P_h u_h)(P_h w)(u-u_h)- \int_\Omega uw(u-u_h) \\
    &\le \max\{-w_\ell,w_u\}(\|u_h\|_{L^2(\Omega)} + \|u\|_{L^2(\Omega)})\|u - u_h\|_{L^2(\Omega)} \\
    &\le \max\{-w_\ell,w_u\}(c_1^{-1} + c_3^{-1})\|f\|_{L^2(\Omega)}
    \big(C_{3/2}^a(w) d_h^{3/2} + C_{3/2}^b d_h^2\big)
    \end{aligned}
    \end{equation*}
    with $c_1$ and $c_3$ from the application of
    Lax--Milgram to \eqref{eq : elliptic_pde}
    and \eqref{eq : ocpe_proj} above.
    \end{proof}

\begin{remark}
Since the diameter $d_h$ of $\{\calQ_h\}_h$ is
in $\Theta(h)$ if $h$ is the mesh size, it holds
$u_h \to u$ in $H^1_0(\Omega, \R)$ as $h \searrow 0$.
\end{remark}

\begin{remark}
In Lemma 3.1 in \cite{manns}, an asymptotically better error estimate $\|u - u_h\|_{L^2(\Omega)} = O(h^2)$ is shown for the case $\Omega \subset \R$, that
is, a one-dimensional domain. Their proof requires a uniform $L^\infty$-estimate on $\nabla u_h$, which is unavailable in our setting even if $A$ is
(smooth enough) such that $u_h$ admits $H^2$-regularity.
\end{remark}
\subsection{Verification of the assumptions for \texorpdfstring{\eqref{eq : ocpc}}{OCPH}}\label{subsec : verification_of_ass_helmholtz}
The weak formulation of the PDE constraining \eqref{eq : ocpc} is that $u \in H^1(\Omega,\C)$
solves
\begin{multline} \label{eq : complex_pde}
a[u,v] \coloneqq \int_{\Omega} \nabla u \cdot \nabla \conj{v} - k_0^2 \int_{\Omega} (1 + qw) u \conj{v} - ik_0 \int_{\partial \Omega} u \conj{v}\\
= k_0^2 \int_{\Omega} qw u_0 \conj{v}  \text{ for all } v \in H^1(\Omega, \C).
\end{multline}
Since $q = \tilde{q} 1_{\tilde{\Omega}}$ for some $\tilde{\Omega} \subset \Omega$
(see \cref{sec:intro}) and $w_\ell < w_u$, we have $1 + \tilde{q} \min\{0,w_\ell\}
\le 1 + qw \le 1 +\tilde{q} \max\{0,w_u\}$ a.e.\ for all $w \in \wad$.
This implies that we can use standard elliptic PDE theory to deduce the existence
of solutions under a sufficient condition on the values of $w_u$, $\tilde{q}$,
and $k_0$.

\begin{proposition}\label{prp:helmholtz_existence_uniqueness}
Let the scalars $w_u \in \R$, $\tilde{q} > 0$, $k_0 > 0$ satisfy
\begin{gather}\label{eq:helmholtz_elliptic_sufficient}
\frac{\min\{1,k_0\}}{\sqrt{2}}\lambda > k_0^2 (1 + \tilde{q}\max\{0, w_u\})
\enskip
\Leftrightarrow
\enskip
\frac{\lambda - \sqrt{2}\max\{k_0,k_0^2\}}{
	\sqrt{2}\tilde{q}\max\{k_0,k_0^2\}} > \max\{0, w_u\},
\end{gather}
where $\lambda$ is the norm equivalence constant from \cref{sec:notation}.
Then \eqref{eq : complex_pde} admits a unique solution for all $w \in \wad$
with
\[
\|u\|_{H^1(\Omega,\C)}
\le c(w_u, \tilde{q}, k_0)^{-1} k_0^2 \tilde{q}\max\{-w_\ell,w_u\}\|u_0\|_{L^2(\Omega,\C)}
\]
with $c(w_u, \tilde{q}, k_0) \coloneqq \lambda\min\{1, k_0\}/\sqrt{2} - k_0^2 (1 + \tilde{q}\max\{0, w_u\})$.

Condition \eqref{eq:helmholtz_elliptic_sufficient} can be satisfied
for all small enough $k_0 > 0$ if $\tilde{q}$ and $w_u$ are given.
\end{proposition}
\begin{proof}
We observe that our sesquilinear form $a$ satisfies
\begin{equation}\label{eq : a_coerc}
\begin{aligned}
|a[u,u]|
&\ge \bigl|
\|\nabla u\|_{L^2(\Omega,\C^n)}^2
- i k_0\|u\|_{L^2(\partial\Omega,\C)}^2\bigr|
- k_0^2\Bigl|
\int_{\Omega}(1 + qw)|u|^2
\Bigr|\\
&\ge
\frac{\min\{1,k_0\}}{\sqrt{2}}
\bigl(
\|\nabla u\|_{L^2(\Omega,\C^n)}^2
+ \|u\|_{L^2(\partial\Omega,\C)}^2\bigr)
- k_0^2(1 + \tilde{q}\max\{0,w_u\})
\|u\|_{L^2(\Omega,\C)}^2
\\
&\ge \frac{\lambda\min\{1,k_0\}}{\sqrt{2}}\|u\|_{H^1(\Omega,\C)}^2
- k_0^2(1 + \tilde{q}\max\{0,w_u\})
\|u\|_{H^1(\Omega,\C)}^2.
\end{aligned}
\end{equation} 
Moreover, $a[u,v]$ is bounded by the trace and the Hölder theorems. Thus, the Lax--Milgram lemma for
complex-valued forms \cite[Theorem 6]{Lax2002} is applicable.
For the right-hand side of the norm estimate in the Lax--Milgram lemma, we observe
\[ \|k_0^2 q wu_0\|_{(H^1(\Omega,\C))^*}
\le \|k_0^2 q wu_0\|_{L^2(\Omega,\C)}
\le k_0^2 \tilde{q} \max\{-w_\ell,w_u\} \|u_0\|_{L^2(\Omega,\C)}
\]
by definition (1st inequality) and H\"older inequality (2nd inequality).
The satisfaction of \eqref{eq:helmholtz_elliptic_sufficient} for all small enough $k_0 > 0$ follows
by inspection.
\end{proof}
\begin{remark}\label{rmk : adjoint_h1}
    Since the estimations \eqref{eq : a_coerc} are independent of the sign of the boundary term in the Helmholtz PDE, the adjoint PDE satisfies \eqref{eq : a_coerc}. 
    The considerations on the source term are also applicable for a more general source term $g \in L^2(\Omega, \C)$. 
    Consequently, $p$, being the solution to the adjoint PDE 
\begin{equation}\label{eq : adjoint}
    \int_{\Omega} \nabla p \cdot \nabla \conj{v} - k_0^2 \int_{\Omega} (1 + qw) p \conj{v} + ik_0 \int_{\partial \Omega} p \conj{v} = \\
\int_{\Omega} g\conj{v} \text{ for all } v \in H^1(\Omega, \C),
\end{equation}  
satisfies 
\[
\|p\|_{H^1(\Omega,\C)}
\le \underbrace{c(w_u, \tilde{q}, k_0)^{-1}}_{\eqqcolon c_2^p} \|g\|_{L^2(\Omega, \C)}.
\]
\end{remark}
We now certify our assumptions one by one.

\paragraph{\Cref{ass : pdes} \ref{ass : wu_bounds_exist}.}
To verify $L^\infty$-bounds on the solution to \eqref{eq : complex_pde}, we first show $H^2$-regularity.
To this end, we make a short bootstrapping argument building on the analysis in \S8.1 in \cite{melenk_diss},
where $H^2$-regularity of a similar PDE is shown. Specifically, we can observe
that the PDE (HeW) in
\cite{melenk_diss} coincides with \eqref{eq : complex_pde} when the adjoint of $a$
is considered and  make the choices $g \coloneqq 0$, $f \coloneqq k_0^2 qw (u + u_0)$ in (HeW) in
\cite{melenk_diss}. Here, we have $f \in L^2(\Omega, \C)$ because $u$, $u_0 \in L^2(\Omega, \C)$ and $q$,
$w \in L^\infty(\Omega)$ hold. We obtain the following corollary.
\begin{corollary}\label{cor:helmholtz_h2_regularity}
Let the scalars $w_u \in \R$, $\tilde{q} > 0$, $k_0 > 0$ satisfy \eqref{eq:helmholtz_elliptic_sufficient}.
Let $u$ solve \eqref{eq : complex_pde} for $w \in \wad$. Then, 
\begin{multline}
|u|_{H^2(\Omega,\C)}
\coloneqq 
\sqrt{\sum_{1 \le i,j \le n} \|\partial_{x_i}\partial_{x_j}u\|_{L^2(\Omega,\C)}^2}
\\
\le c_1(\Omega,k_0,\tilde{q},w_\ell,w_u)\Bigl(
c_2(k_0,\tilde{q},w_\ell,w_u) + 1\Bigr)\|u_0\|_{L^2(\Omega,\C)}
\end{multline}
with
\begin{align}
c_1(\Omega,k_0,\tilde{q},w_\ell,w_u)
&\coloneqq C(\Omega)(1 + |k_0|)k_0^2\tilde{q}\max\{-w_\ell,w_u\}\text{ and } \\
c_2(k_0,\tilde{q},w_\ell,w_u)
&\coloneqq c(w_u, \tilde{q}, k_0)^{-1} k_0^2 \tilde{q}\max\{-w_\ell,w_u\},
\end{align}
where $C(\Omega) > 0$ is the constant in Proposition 8.1.4 in \cite{melenk_diss}.
\end{corollary}
\begin{proof}
Due to uniqueness of the solution, the solution $u$ to the (adjoint of the) PDE (HeW)
in \cite{melenk_diss} for the choices above for $g$ and $f$ must coincide with the
solution to \eqref{eq : complex_pde}. From the second estimate in Proposition 8.1.4
in \cite{melenk_diss}, we obtain again with the choices above for $g$  and $f$ in (HeW)
in \cite{melenk_diss} that
\begin{align*}
|u|_{H^2(\Omega,\C)}
&\le C(\Omega)(1 + |k_0|)\|k_0^2 qw (u + u_0)\|_{L^2(\Omega,\C)} \\
&\le c_1(\Omega,k_0,\tilde{q},w_\ell,w_u)\|u + u_0\|_{L^2(\Omega,\C)} \\
&\le c_1(\Omega,k_0,\tilde{q},w_\ell,w_u)\Bigl(
c_2(k_0,\tilde{q},w_\ell,w_u)\|u_0\|_{(H^1(\Omega,\C))^*}+ \|u_0\|_{L^2(\Omega,\C)}\Bigr)
\end{align*}
holds with $C(\Omega) > 0$ being the constant from Proposition 8.1.4 in \cite{melenk_diss}
and $c_2(k_0,\tilde{q},w_\ell,w_u)$ due to \Cref{prp:helmholtz_existence_uniqueness}.
\end{proof}

We thus obtain the following result that certifies \Cref{ass : pdes} \ref{ass : wu_bounds_exist}.
\begin{proposition}\label{prp:helmholtz_linfty_bounds}
Let the scalars $w_u \in \R$, $\tilde{q} > 0$, $k_0 > 0$ satisfy \eqref{eq:helmholtz_elliptic_sufficient}.
Let $u$ solve \eqref{eq : complex_pde} for $w \in \wad$.
Then,
\[ \|u\|_{L^\infty(\Omega,\C)} \le
\sqrt{2}c_{3}(\Omega)(c_1 + c_1c_2 + c_2)\|u_0\|_{L^2(\Omega,\C)} \eqqcolon u^* \]
holds with $c_1 = c_1(\Omega,k_0,\tilde{q},w_\ell,w_u)$, $c_2 = c_2(k_0,\tilde{q},w_\ell,w_u)$
from \Cref{cor:helmholtz_h2_regularity},
and $c_3(\Omega)$ being the embedding constant of $H^2(\Omega,\C) \hookrightarrow L^\infty(\Omega,\C)$.
\end{proposition}
\begin{proof}
The claim follows from
$\|u\|_{H^2(\Omega,\C)} \le \sqrt{2}(\|u\|_{H^1(\Omega,\C)} + |u|_{H^2(\Omega,\C)})$,
the estimates from \Cref{prp:helmholtz_existence_uniqueness}
and \Cref{cor:helmholtz_h2_regularity} as well as the
Sobolev embedding $H^2(\Omega,\C) \hookrightarrow L^\infty(\Omega)$,
which follows, e.g., from Case A in Theorem 4.12 in \cite{adams}.
\end{proof}
Consequently, we can set $\Im u_u =  \Re u_u \ge
u^*$
and $\Im u_\ell = \Re u_\ell \le -u^*$ to satisfy \Cref{ass : pdes} \ref{ass : wu_bounds_exist}.
\begin{remark}\label{rmk : adjoint_linf}
    Since Proposition 8.1.4 in \cite{melenk_diss} holds for the adjoint PDE \eqref{eq : adjoint}, we similarly obtain 
\[
 \|p\|_{L^\infty(\Omega,\C)} \le
\sqrt{2}c_{3}(\Omega)(C(\Omega)(1+|k_0|) + c_1c_2^p + c_2^p)\|g\|_{L^2(\Omega, \C)}.
\]
for $c_1 = c_1(\Omega,k_0,\tilde{q},w_\ell,w_u)$
from \Cref{cor:helmholtz_h2_regularity}, $c_2^p$ from \Cref{rmk : adjoint_h1}, and $c_3(\Omega)$ being the embedding constant of $H^2(\Omega,\C) \hookrightarrow L^\infty(\Omega,\C)$.
\end{remark}
\paragraph{\Cref{ass : pdes} \ref{ass : phwphu_pdesolexists}.} 
In the locally-averaged variant of \eqref{eq : ocpc}, we seek $u_h \in H^1(\Omega,\C)$ that solves
\begin{multline} \label{eq:helmholtz_locally_averaged}
a_h[u_h,v] \coloneqq \int_{\Omega} \nabla u_h \cdot \nabla \conj{v} - k_0^2 \int_{\Omega} u_h \conj{v} 
- k_0^2\int_{\Omega} q (P_h w) (P_h u_h) \conj{v} - ik_0 \int_{\partial \Omega} u_h \conj{v}\\
= k_0^2 \int_{\Omega} q w u_0 \bar{v}  \text{ for all } v \in H^1(\Omega, \C).
\end{multline}
We assume the following compatibility between the discretization
$\calQ_{h}$ and the subdomain $\tilde{\Omega}$:
\begin{gather}\label{eq:tildeOmega_compatbility}
\tilde{\Omega} = \bigcup_{i \in I} Q_h^i\enskip \text{ for some } I \subset \{1,\ldots,N_h\}.
\end{gather}
Importantly, \eqref{eq:tildeOmega_compatbility} implies $P_h(qw) = q P_h w$.
Then, similar to above, since $q = \tilde{q} 1_{\tilde{\Omega}}$ for some $\tilde{\Omega} \subset \Omega$
(see \cref{sec:intro}) and $w_\ell < w_u$, we have $\tilde{q} \min\{0,w_\ell\}
\le q (P_h w) \le \tilde{q} \max\{0,w_u\}$ a.e.\ for all $w \in \wad$.
Again, we can use standard elliptic PDE theory to deduce the existence of solutions under a sufficient condition
on $w_u$, $\tilde{q}$, $k_0$, and now also $w_\ell$.

\begin{proposition}\label{prp:helmholtz_locally_averaged_existence_uniqueness}
	Let the scalars $w_\ell$, $w_u \in \R$, $\tilde{q} > 0$, $k_0 > 0$ satisfy $w_\ell < w_u$ and
	\begin{gather}\label{eq:helmholtz_locally_averaged_elliptic_sufficient}
	\frac{\min\{k_0,1\}}{\sqrt{2}}\lambda > k_0^2 (1 + \tilde{q}\max\{-w_\ell, w_u\})
	\enskip
	\Leftrightarrow
	\enskip
	\frac{\lambda - \sqrt{2}\max\{k_0,k_0^2\}}{
		\sqrt{2}\tilde{q}\max\{k_0,k_0^2\}} > \max\{-w_\ell, w_u\},
	\end{gather}
	where $\lambda$ is the norm equivalence constant from \cref{sec:notation}.
	Then \eqref{eq:helmholtz_locally_averaged} admits a unique
	solution $u_h$ for all $w \in \wad$	with
	\[
	\|u_h\|_{H^1(\Omega,\C)}
	\le \bar{c}(w_\ell, w_u, \tilde{q}, k_0)^{-1} k_0^2 \tilde{q}\max\{-w_\ell,w_u\}\|u_0\|_{(H^1(\Omega,\C))^*}
	\]
	with $\bar{c}(w_\ell, w_u, \tilde{q}, k_0) \coloneqq \lambda\min\{k_0,1\}/\sqrt{2} - k_0^2 (1 + \tilde{q}\max\{-w_\ell, w_u\})$.
	
	Condition \eqref{eq:helmholtz_locally_averaged_elliptic_sufficient}  can be satisfied
	for all small enough $k_0 > 0$ if $\tilde{q}$,
	$w_\ell$, and $w_u$ are given. It implies
	\eqref{eq:helmholtz_elliptic_sufficient}.
\end{proposition}
\begin{proof}
With a similar argument as in \Cref{prp:helmholtz_existence_uniqueness}, the sesquilinear form $a_h$ satisfies
\begin{equation}\label{eq : ah_coerc}
\begin{aligned}
|a_h[u_h,u_h]|
&\ge \bigl|
\|\nabla u_h\|_{L^2(\Omega,\C^n)}^2
- i k_0\|u_h\|_{L^2(\partial\Omega,\C)}^2\bigr|
- k_0^2\Bigl|\int_{\Omega}|u_h|^2 + \int_{\Omega} q(P_hw) (P_hu_h)u_h\Bigr|\\
&\ge
\frac{\lambda\min\{1,k_0\}}{\sqrt{2}}\|u_h\|_{H^1(\Omega,\C)}^2
- k_0^2(1 + \tilde{q}\max\{-w_\ell,w_u\})
\|u_h\|_{L^2(\Omega,\C)}^2,
\end{aligned}
\end{equation}
where we have used the nonexpansiveness of $P_h : L^2(\Omega) \to L^2(\Omega)$.
The remainder of the proof works analogously to the one of \Cref{prp:helmholtz_existence_uniqueness}.
\end{proof}
As a corollary, we obtain the following result, which yields
$L^\infty$-bounds on the solution to
\eqref{eq:helmholtz_locally_averaged}.
\begin{proposition}
\label{prp:helmholtz_locally_averaged_linfty_bounds}
Let the scalars $w_\ell < w_u \in \R$, $\tilde{q} > 0$, $k_0 > 0$ satisfy \eqref{eq:helmholtz_locally_averaged_elliptic_sufficient}.
Let $u_h$ solve \eqref{eq:helmholtz_locally_averaged} for $w \in \wad$.
Then,
\[ \|u_h\|_{L^\infty(\Omega,\C)} \le
\sqrt{2}c_{3}(\Omega)(c_1 + c_1\bar{c}_2 + \bar{c}_2)\|u_0\|_{L^2(\Omega,\C)}
\eqqcolon u_h^*
 \]
holds with $c_1 = c_1(\Omega,k_0,\tilde{q},w_\ell,w_u)$
from \Cref{cor:helmholtz_h2_regularity},
$\bar{c}_2 \coloneqq \bar{c}(w_\ell, w_u, \tilde{q}, k_0)^{-1} k_0^2 \tilde{q}\max\{-w_\ell,w_u\}$
with $\bar{c}(w_\ell, w_u, \tilde{q}, k_0)$
from \Cref{prp:helmholtz_locally_averaged_existence_uniqueness},
and $c_3(\Omega)$ being the embedding constant of $H^2(\Omega,\C) \hookrightarrow L^\infty(\Omega,\C)$.
\end{proposition}
\begin{proof}
Follows along the same lines as
\Cref{cor:helmholtz_h2_regularity}
and
\Cref{prp:helmholtz_linfty_bounds}.
\end{proof}

\paragraph{\Cref{ass : pdes} \ref{ass : richness_of_feasible_set}.}
The non-expansiveness of $P_h : L^\infty(\Omega, \C) \to L^\infty(\Omega, \C)$
yields that we can use bounds
on $\|u_h\|_{L^\infty(\Omega, \C)}$ for
$\|P_h u_h\|_{L^\infty(\Omega, \C)}$.
Consequently, by means of
\Cref{prp:helmholtz_locally_averaged_linfty_bounds}, we can set
$\Im u_u^i \coloneqq  \Re u_u^i \ge u_h^*$ and
$u_\ell^i \le -u_h^*$
for all $i \in \{1,\ldots,N_h\}$
to obtain valid bounds.

\paragraph{\Cref{ass : pdes} \ref{itm:ubnd_limits}.}
Since $P_h$ is non-expansive, $\|P_hu_h\|_{L^\infty(\Omega,\C)} \le \|u_h\|_{L^\infty(\Omega,\C)}$
holds. Consequently, we can choose 
\[ u_{\ell,h} = u_\ell \le \min\{-u^*, -u_h^*\}
\quad\text{and}\quad
u_{u,h} = u_u \ge \min\{u^*, u_h^*\}
\]
to certify the assumption while still complying with
\Cref{ass : pdes} \ref{ass : wu_bounds_exist}, \ref{ass : richness_of_feasible_set}. 

\paragraph{\Cref{ass : pdes} \ref{itm:uuh_approx}.}
Let $d_h$ denote the maximal diameter of the grid cells
of the $h$-grid. To this end, we use the a priori 
estimate on the local averaging error
\begin{equation}\label{eq : apriori_helm}
    \|u-u_h\|_{L^2(\Omega, \C)}
    \leq  \tilde{C}_{3/2}^a(w) d_h^{3/2} + \tilde{C}_{3/2}^b d_h^{2}
\end{equation}
for some positive constants $\tilde{C}_{3/2}^a(w) =
\tilde{C}_{3/2}(\Omega,k_0,\tilde{q},w_\ell,w_u, u_0,w)$, $\tilde{C}_{3/2}^b = \tilde{C}_{3/2}^b(\Omega,k_0,\tilde{q},w_\ell,w_u, u_0)$,
which will be shown in \Cref{prp:proof_apriori_helm}.

Using \eqref{eq : apriori_helm}, the certification of \Cref{ass : pdes} \ref{itm:uuh_approx} works as follows. First, we test \eqref{eq : complex_pde}
and \eqref{eq:helmholtz_locally_averaged} with $(u-u_h)$:
\begin{equation*}
    \begin{aligned}
    \int_\Omega  \nabla (u-u_h) \cdot \conj{\nabla(u-u_h)} - ik_0 \int_{\partial \Omega} (u-u_h) \conj{(u-u_h)}= 
        k_0^2\int_\Omega u(1+qw)\conj{(u-u_h)} \\+ k_0^2\int_\Omega (P_h u_h) (1+q P_hw) \conj{(u-u_h)}. 
    \end{aligned}
    \end{equation*}
Then, we estimate both hand sides separately. 
We estimate the left-hand side from below as in \eqref{eq : a_coerc} to obtain
\begin{equation}
    \begin{aligned}
        \bigr| \|\nabla (u-u_h)\|_{L^2(\Omega,\C^n)}^2
    - i k_0\|u-u_h\|_{L^2(\partial\Omega,\C)}^2\bigr|
    \end{aligned}
    \end{equation} 
For the right-hand side, we use the Hölder inequality together with 
the bounds on $u$ from \Cref{prp:helmholtz_linfty_bounds}
and $P_h u_h$ from \Cref{prp:helmholtz_locally_averaged_linfty_bounds}.
We obtain
\begin{equation*}
    \begin{aligned}
    \|u-u_h\|_{H^1(\Omega, \C)}^2\leq 
    k_0^2 \frac{2}{\lambda\min\{1,k_0\}} c_{3}(\Omega)(c_1 + c_1\bar{c}_2 + c_1c_2 + c_2+ \bar{c}_2)\|u_0\|_{L^2(\Omega,\C)}  \\
    \cdot \bigr(1+\tilde{q} \max\{-w_\ell,w_u\}\bigr) 
    \|u - u_h\|_{L^2(\Omega,\C)}.
\end{aligned}
\end{equation*}
Then the assertion follows from \eqref{eq : apriori_helm}.

\begin{remark}
    Since the diameter $d_h$ of $\{\calQ_h\}_h$ is
    in $\Theta(h)$ if $h$ is the mesh size, it holds
    $u_h \to u$ in $H^1(\Omega, \C)$ as $h \searrow 0$.
\end{remark}

\begin{proposition}\label{prp:proof_apriori_helm}
    Let the scalars $w_\ell$, $w_u \in \R$, $\tilde{q} > 0$, $k_0 > 0$ with $w_\ell < w_u$ satisfy \eqref{eq:helmholtz_locally_averaged_elliptic_sufficient}. Let $u$ and $u_h$ solve \eqref{eq : complex_pde} and \eqref{eq:helmholtz_locally_averaged} for $w \in \wad$
    respectively.  Then, the a priori estimate \eqref{eq : apriori_helm} is satisfied with
    \begin{equation*}
    \begin{aligned}\tilde{C}_{3/2}^a (w)
     &= k_0^2\alpha(\Omega)^{\frac{3}{2}}\tilde{q} \|u_0\|_{L^2(\Omega,\C)} |w_u-w_\ell|^{\frac{1}{2}} \TV(w)^{\frac{1}{2}} \\
     &\quad \quad\Bigr(c_2^p + \sqrt{2} c_3\big(C(\Omega)(1+|k_0|) + c_1c_2^p + c_2^p\big)\Bigr),\\
     \tilde{C}_{3/2}^b &=   \alpha(\Omega)^2 c_2 c_2^p (1+\tilde{q} \max\{-w_\ell,w_u\})  \|u_0\|_{L^2(\Omega,\C)}
    \end{aligned}
    \end{equation*}
    with $c_1, c_2$ from \Cref{cor:helmholtz_h2_regularity}, $c_3$ being the embedding constant of $H^2(\Omega,\C) \hookrightarrow L^\infty(\Omega,\C)$, $c_2^{u_h}$ from \Cref{prp:helmholtz_locally_averaged_linfty_bounds}, $c_2^p$ from \eqref{eq : adjoint}, $\alpha(\Omega)$ from \eqref{eq : estim_ph}.
\end{proposition}
\begin{proof}
Similar to \Cref{thm: apriori}, we use the use the Aubin--Nitsche duality argument and the identification
\begin{equation*}
    \begin{aligned}
   \|u-u_h\|_{L^2(\Omega, \C)} &= \underset{g \in L^2(\Omega, \C), g \neq 0}{\sup} 
    \frac{|(u-u_h,g)_{L^2(\Omega, \C)}|}{\|g\|_{L^2(\Omega, \C)}}.
    \end{aligned}
\end{equation*}
We estimate $(u- u_h,g)_{L^2(\Omega, \C)}$ 
for every $0 \neq g \in L^2(\Omega, \R)$ with the help of the adjoint 
formulation for \eqref{eq : complex_pde} and insertion of a suitable zero. 
In particular, we will use that
\begin{equation} \label{eq : galerkin}
    \int_\Omega \rho (\phi - P_h \phi)(P_h \psi)(P_h \theta) 
    = \sum_{i\in I} \tilde{\rho }
    \left(\frac{1}{\meas Q^i_h} \int_{Q^i_h} \psi\right) 
    \left( \frac{1}{\meas Q^i_h} \int_{Q^i_h}\theta\right)
    \underbrace{\int_{Q^i_h} \phi - P_h \phi}_{= 0}
    = 0.
\end{equation}
holds for all $\phi, \psi, \theta \in \L, \rho = \tilde{\rho } 1_{\tilde{\Omega}_1}$ 
with $\tilde{\Omega}_1$ satisfying the compatibility assumption \eqref{eq:tildeOmega_compatbility} for some $I \subset \{1,\ldots, N_h\}$.
Our goal is to
rewrite $(\tilde{u}- \tilde{u}_h,g)_{L^2}$ as a sum of terms each
containing multiple factors of a type
\begin{equation}\label{eq : aux_diff}
    v - P_h v \text{ for } v \in L^1(\Omega)
\end{equation}
so that we can apply \eqref{eq : estim_ph} and deduce an estimate
with respect to $d_h$.

Let $p$ be the solution to the adjoint \eqref{eq : adjoint}.
We note that the adjoint \eqref{eq : adjoint} differs from the PDE \eqref{eq : complex_pde} by the sign of the boundary term and the source term. 
We test \eqref{eq : complex_pde} and \eqref{eq:helmholtz_locally_averaged} with $p$, subtract the latter from the former, and insert 
a suitable zero to obtain 
\begin{equation*}
\int_\Omega  \nabla (u-u_h) \cdot \nabla\conj{p} = 
    k_0^2\int_\Omega (u-P_hu_h)(1+qw)\conj{p} + k_0^2\int_\Omega (P_h u_h) q(w - P_hw) \conj{p} + ik_0 \int_{\partial \Omega} (u-u_h) \conj{p}
\end{equation*}
Next, we test \eqref{eq : adjoint} with $u-u_h$ and use the conjugation of
the above equation to rewrite the term $\int_{\Omega} \nabla \overline{(u - u_h)}
\cdot \nabla p$ and get 
\begin{equation*}
\begin{aligned}
(g, u-u_h)_{L^2(\Omega, \C)} &= \Bigg(k_0^2\int_\Omega \conj{(u-P_hu_h)(1+qw)} p + k_0^2\int_\Omega \conj{(P_h u_h) q(w - P_hw)} p
\\
&\quad\quad - ik_0 \int_{\partial \Omega} \conj{(u-u_h)} p \Bigg)
                 - k_0^2 \int_\Omega p (1+qw) \conj{(u - u_h)} 
\\
&\quad\quad                  
+ ik_0 \int_{\partial \Omega}  p \conj{(u-u_h)}.
\end{aligned}
\end{equation*}
Note that the conjugation implied the negative sign of the boundary term so that the boundary terms on the right-hand side vanish. 
This together with the combination of the first and the forth terms implies 
\begin{equation*}
    \begin{aligned}
        (g, u-u_h)_{L^2(\Omega, \C)}  = k_0^2\int_\Omega \conj{(u_h-P_hu_h)}(1+qw) p + k_0^2\int_\Omega \conj{(P_h u_h) }q(w - P_hw) p.
    \end{aligned}
\end{equation*}
We now insert a suitable zero to obtain 
\begin{equation*}
    \begin{aligned}
        (g, u-u_h)_{L^2(\Omega, \C)}  = k_0^2\int_\Omega \conj{(u_h-P_hu_h)}q(w-P_hw) p + k_0^2\int_\Omega \conj{(u_h-P_hu_h)}(1+ qP_hw) p \\
                + k_0^2\int_\Omega \conj{(P_h u_h)} q(w - P_hw) p.
    \end{aligned}
\end{equation*}
Inserting two zeros of the form \eqref{eq : galerkin} yields 
\begin{equation*}
    \begin{aligned}
        (g, u-u_h)_{L^2(\Omega, \C)} = k_0^2\int_\Omega \conj{(u_h-P_hu_h)}q(w-P_hw) p + k_0^2\int_\Omega \conj{(u_h-P_hu_h)}(1+ qP_hw)(p-P_hp)  \\
                + k_0^2\int_\Omega \conj{(P_h u_h)} q(w - P_hw) (p-P_h p).
    \end{aligned}
\end{equation*}
Now we employ the Hölder inequality together with \eqref{eq : estim_ph} 
to estimate each of the factors and obtain 
\begin{equation*}
    \begin{aligned}
        (g, u-u_h)_{L^2(\Omega, \C)} \leq 
        &k_0^2 \Big( \alpha(\Omega) d_h \|\nabla u_h\|_{L^2(\Omega, \C^n)} 
                                    \tilde{q} \sqrt{\alpha(\Omega)|w_u-w_\ell| d_h \TV(w)}
                                    \|p\|_{L^\infty(\Omega, \C)} \\
                                    + &\alpha(\Omega) d_h \|\nabla u_h\|_{L^2(\Omega, \C^n)} 
                                    (1+\tilde{q}\|P_h w\|_{L^\infty(\Omega)})
                                    \alpha(\Omega) d_h \|\nabla p\|_{L^2(\Omega, \C^n)} 
                                    \\
                                    + &\|P_h u_h\|_{L^\infty(\Omega, \C)}
                                    \tilde{q} \sqrt{\alpha(\Omega)|w_u-w_\ell| d_h \TV(w)}
                                    \alpha(\Omega) d_h \|\nabla p\|_{L^2(\Omega, \C^n)} \Big) ,
    \end{aligned}
\end{equation*}
where we have used $\|w - P_hw\|_{L^2(\Omega)} \le \sqrt{(w_u - w_\ell)\|w - P_hw\|_{L^1(\Omega)}}$.

The usage of \Cref{prp:helmholtz_existence_uniqueness}, \Cref{rmk : adjoint_h1}, 
\Cref{prp:helmholtz_locally_averaged_existence_uniqueness}, \Cref{prp:helmholtz_linfty_bounds}, 
\Cref{prp:helmholtz_locally_averaged_linfty_bounds}, \Cref{rmk : adjoint_linf}
and division by $\|g\|_{L^2(\Omega, \C)}$ on both hand sides yields the claimed bound.  
\end{proof}
\subsection{Approximation of Total Variation}\label{subsec : verification_of_ass_tv}
On one-dimensional domains $\Omega \subset \R$, the total variation of a given function can be approximated by the total variation of approximating
step functions, that is, the functionals $\TV \circ P_\tau$ $\Gamma$-converge to the functional $\TV$. This fact is used in \cite{manns}.
On multi-dimensional domains $\Omega \subset \R^n$, $n \ge  2$, a gap may persist between $\TV \circ P_\tau$ and $\TV$ as $\tau \searrow 0$;
see, e.g., Example 2.10 in \cite{schiemann2025discretization}. Consequently, we use a discretization of the dual formulation of $\TV$ by means of a
Raviart--Thomas finite-element ansatz as introduced in \cite[(3.4)]{corentin} for $n \ge 2$. This choice will satisfy \cref{ass : bv}. In this section, we restrict to rectangular domains and make the more specific assumption (cf.\ \cref{ass : grid_cells})
that the discretization grid for $w$ is an
axis-aligned grid of quadrilaterals and that
a bounded eccentricity condition is satisfied
when a sequence of these grids with vanishing mesh size
is considered. This is formalized below.
\begin{assumption}\label{ass:tau_grids_for_tv}
For the domain $\Omega$ and the discretization grids
(partitions of $\Omega$) $\tilde{Q}_{\tau} = \{ \tilde{Q}_\tau^1,\ldots,\tilde{Q}_\tau^{N_\tau}\}$
as introduced in \cref{sec : mcc_theory}, we
assume the following.
\begin{enumerate}
\item $\Omega = (a_1,b_1) \times \ldots \times (a_n,b_n)$.
\item For all $\tilde{Q}^j \in \tilde{\calQ}_{\tau}$, it holds $\tilde{Q}^j = I_1 \times \ldots \times I_n$ for bounded intervals $I_1$, $\ldots$, $I_n \subset \R$.
\item We assume that the grids $\tilde{\calQ}_{\tau}$
satisfy a bounded eccentricity condition analogously
to \cref{ass : grid_cells} \ref{itm:bounded_eccentricity}
as $\tau \searrow 0$.
\end{enumerate}
\end{assumption}
We define
\begin{gather}\label{eq:dfn_tv_tau}
\TVt(w) \coloneqq \sup \left\{ \int_{\Omega} w \div\phi : \phi \in \RT, \|\phi\|_{L^\infty(\Omega)} \leq 1 \right\}
\end{gather}
for $w \in L^2(\Omega)$, where $RT0^\tau_0 \subset \Hzerodiv(\Omega)$ 
denotes the Raviart--Thomas fields of the lowest order with 
vanishing normal trace.

\paragraph{\Cref{ass : bv} \ref{itm:consistency}.} To prove the claims, we use the smooth approximation property
\begin{gather}\label{eq:h0div_smooth_approx}
\overline{C^{1}_C(\Omega,\R^n)}^{\Hzerodiv} = \Hzerodiv(\Omega),
\end{gather}
which is proven in \cite{hintermuller2015density}.
\begin{lemma}
For all $\tau$, the functional $\TV_\tau : L^2(\Omega) \rightarrow [0, \infty)$ is lower semi-continuous wrt.\ $L^1$-convergence,  
$\TVt \circ P_\tau \leq \TV$, and $\TVt \circ P_\tau \leq \TVt$.
\end{lemma}
\begin{proof}
The lower semi-continuity follows analogously to the lower semi-continuity of $\TV$ as shown in Remark 3.5 and Proposition 3.6
in \cite{ambrosio2000} with an approximation of the Raviart--Thomas field by a continuously differentiable field, which is possible
by virtue of \eqref{eq:h0div_smooth_approx}.

For the second and third claim, we first observe that $\TVt = \TVt \circ P_\tau$. Indeed, for $\phi \in \RT$ we have that $\div \phi$
is constant per grid cell $\tilde{Q}_\tau^j$, $j\in \{1,\ldots,N_{\tau}\}$, which implies
\begin{align*}
\int_\Omega w \div \phi &=\sum_{j = 1}^{N_\tau} \int_{\tilde{Q}_\tau^j} w \div \phi
=\sum_{j = 1}^{N_\tau} \frac{1} {\meas (\tilde{Q}_\tau^j)} \int_{\tilde{Q}_i} \div \phi \int_{\tilde{Q}_i} w \\
&=\sum_{j = 1}^{N_\tau}n \frac{1} {\meas (\tilde{Q}_\tau^j)} \int_{\tilde{Q}_\tau^j} \div \phi (P_\tau w)
=\int_\Omega \div \phi (P_\tau w),
\end{align*}
yielding $\TVt = \TVt \circ P_\tau$.
Then the second and third claim follow from
\begin{gather*}
\begin{aligned}
\TVt(P_\tau(w)) = \TVt(w)
&\le \sup\left\{ \int_{\Omega} \div \phi w : \phi \in \Hzerodiv(\Omega), \|\phi\|_{L^\infty(\Omega, \R^n)} \leq 1 \right\} \\
&\underset{\mathclap{\eqref{eq:h0div_smooth_approx}}}=\sup\left\{ \int_{\Omega} \div \phi w : \phi \in C_C^1(\Omega, \R^n), \|\phi\|_{L^\infty(\Omega, \R^n)}  \leq 1 \right\} = \TV(w)
\end{aligned}
\end{gather*}
for all $w \in L^2(\Omega)$.
\end{proof}

\paragraph{\Cref{ass : bv} \ref{itm:TVh_gamma}.}
The $\Gamma$-convergence is a straightforward consequence 
of existing results under \cref{ass:tau_grids_for_tv}.
\begin{lemma}
Let \cref{ass:tau_grids_for_tv} hold.
$\TV_\tau \circ P_{\tau}$ $\Gamma$-converges to $\TV$ for $\tau \searrow 0$
on $L^2(\Omega)$ with respect to $L^1$-convergence.
\end{lemma}
\begin{proof}
The $\lim \sup$-inequality follows from \Cref{ass : bv}
\ref{itm:consistency} with the choice $f_\tau \coloneqq f$,
yielding $\TVt (P_\tau f_\tau) =  \TVt (P_\tau f)\leq \TV(f)$.
For the $\lim\inf$-inequality, let $f \in L^1(\Omega)$ and $f_\tau \to f$ 
in $L^1(\Omega)$ for $\tau \searrow 0$. Then $P_\tau(f_\tau) \to f$
in $L^1(\Omega)$
by means of the non-expansiveness of $P_\tau$ and the Lebesgue's differentiation theorem. Then, Theorem 2.9
in \cite{schiemann2025discretization} implies the $\lim\inf$-inequality.
\end{proof}

\paragraph{\Cref{ass : bv} \ref{itm:TVh_compactness}.}
We give a lengthy proof of \Cref{ass : bv} in several steps, where we settle an open question (see Remark 3.15 in
\cite{schiemann2025discretization}) as a corollary of our first auxiliary result.
The first step is to overestimate the total variation of a piecewise constant function on the $\tau$-grid $\tilde{\calQ}_\tau \coloneqq \{\tilde{Q}_\tau^1,\ldots,\tilde{Q}_\tau^{N_\tau}\}$.

\begin{lemma}\label{lem:sqrtd_estimate_for_dg0}
Let \cref{ass:tau_grids_for_tv} hold.
Let $n \in \{1,2,3\}$.
Let $w \in L^2(\Omega)$ be a piecewise constant function on the $\tau$-grid,
that is, $w = \sum_{j=1}^{N_\tau}\chi_{\tilde{Q}^j_\tau} w_j$.
Then
\[ \TV(w) \le \sqrt{n} \TVt(w). \]
\end{lemma}
\begin{proof}
Our proof relies on a construction of a vector field $\phi \in \RT$ such that
\begin{gather}\label{eq:phi_eq_tv}
\TV(w) = \int_{\Omega} \div\phi w
\quad\text{and}\quad
\|\phi\|_{L^\infty} \le \sqrt{n}
\end{gather}
hold. Dividing by $\sqrt{n}$ then implies
\[ \frac{1}{\sqrt{n}}\TV(w) 
\underset{\mathclap{\eqref{eq:phi_eq_tv}}}= 
\int_{\Omega} \div\Big(\frac{\phi}{\sqrt{n}}\Big) w 
\underset{\mathclap{\eqref{eq:dfn_tv_tau}}}\le \TVt(w),
\]
which proves the claim. We provide the construction of the field $\phi$ only for the case $n=3$ but the claim can
be shown in lower dimensions by following the same idea. 

We provide the detailed construction for our  special case 
that $\Omega = (a_1,b_1)\times \ldots \times (a_n,b_n)$ 
and that our $\tau$-grid is a set of $n$-dimensional
axis-aligned hyperrectangles. Consequently, we can
write each $\tilde{Q}^j \in \tilde{\calQ}_\tau$ as the
Cartesian product of $n$ intervals, that is,
\[ \tilde{Q}^j 
   = I_1^{k_1} \times \ldots \times I_n^{k_n} 
   \eqqcolon \tilde{Q}^{k_1,\ldots,k_n}
\]
for some $k_i \in \{1,\ldots,N_i\}$, $N_i \in \N$,
for $i \in \{1,\ldots,n\}$ and where the
disjoint intervals $I_i^{k}$ are indexed in ascending 
order  in the sense that $\sup I_i^{k-1} = \inf I_i^{k}$ 
for all $k \in \{1,\ldots,N_i\}$ for all $i \in \{1,\ldots,n\}$.

\textbf{Main idea.} The main idea behind our construction
is the following. Since $w$ is piecewise 
constant, $\TV(w)$ is the sum of the facet areas between 
different level sets of $w$ weighted by the respective
jumps height between the level sets over them; see, for 
example, Lemma 2.1 in \cite{manns2023on}. Moreover,
the individual coordinate functions of $\RT$-fields
$\phi$ in our rectangular setting are piecewise affine,
globally continuous one-dimensional functions with trace zero,
that is, sums of so-called hat functions. We thus
construct a coordinate function $\phi_i$,
$i \in \{1,\ldots,d\}$ such that it is a sum of hat
functions with values between $-1$ and $1$
that capture all jumps of $w$ along the
$i$-th coordinate when the other coordinates are kept fixed.
Then the coordinate functions that are constructed in 
this way are shown to yield a vector field $\phi$
in $\RT$ with $\|\phi\|_{L^\infty} \le \sqrt{n}$
and satisfy \eqref{eq:phi_eq_tv}.

\textbf{Coordinate function construction.} Since the construction works completely analogous for all $\phi_i$, 
$i \in \{1,\ldots,n\}$, we only perform the construction of $\phi_1$. To this end, we fix $k_2,\ldots,k_n$ and
abbreviate the sets $\tilde{Q}_m \coloneqq  \tilde{Q}^{m,k_2,\ldots,k_n}$ for
$m \in \{0, \ldots, N_1 - 1\}$, which form a \emph{tower of hyperrectangles} along the first coordinate.
We also abbreviate $w_m \coloneqq \frac{1}{\meas \tilde{Q}} \int_{\tilde{Q}_m} w \equiv w|_{\tilde{Q}_m}$
because $w$ is piecewise constant on $\tilde{\calQ}_{\tau}$.
We define $\phi_1^{k_2,\ldots,k_n} \coloneqq \sum_{m=1}^{N_1-1}\psi^m$,
where $\psi^m : \Omega \to \R$ is a hat function along the first coordinate that captures the potential jump of
$w$ between $\tilde{Q}_m$ and $\tilde{Q}_{m+1}$; see \cref{fig : phi1}. Using that $\sgn(w_m - w_{m+1})$
reflects the change of $w$ at the interface $\overline{\tilde{Q}_m} \cap \overline{\tilde{Q}_{m+1}}$
for all $m \in \{1,\ldots,N_1 - 1\}$, $\psi^m$ is defined by
\[
\psi^m \coloneqq \sgn(w_m - w_{m+1}) \big(1_{\overline{\tilde{Q}_m} \cap \overline{\tilde{Q}_{m+1}}}
+\tfrac{z - x_m}{\ell_m}  
1_{\mathring{\tilde{Q}}_m} - \tfrac{z-x_{m+2}}{\ell_{m+1}} 1_{\mathring{\tilde{Q}}_{m+1}} \big),
\]
where $x_{m} \coloneqq \inf \{ y_1 : (y_1,\ldots,y_n) \in \tilde{Q}_m \}$ for all $m \in \{1,\ldots,N_1\}$,
$x_{N_1} \coloneqq \sup \{y_1 : (y_1,\ldots,y_n) \in \tilde{Q}_{N_1}\}$, and  $\ell_m = x_{m+1} - x_{m}$
for all $m \in \{1,\ldots,N_1\}$. By construction, we have
\begin{gather}\label{eq : phim_bounded}
	\|\psi^m\|_{L^\infty(\Omega)} \le 1
\end{gather}
and, for the weak derivative (note that $\psi^m$ ist not differentiable everywhere) in the first argument,
\begin{gather}\label{eq:phim_der}
\partial_1 \psi^m 
= \sgn(w_m - w_{m+1})\big(\ell_m^{-1} 1_{\mathring{\tilde{Q}}_m} - \ell_{m+1}^{-1} 1_{\mathring{\tilde{Q}}_{m+1}}\big)
\text{ a.e.\ in }\Omega.
\end{gather}
Moreover, the support of each $\psi^m$ is $\overline{\tilde{Q}_m \cup \tilde{Q}_{m+1}}$ and we obtain
\begin{gather} \label{eq : phi1}
\phi_1^{k_2,\ldots,k_n} =
\begin{cases*}
\psi^{m-1} + \psi^m 
&in $\overline{\tilde{Q}_{m}}$ for $m \in \{2, \ldots, N_1-1\}$,\\
\psi^1 &in $\mathring{\tilde{Q}}_{1}$, \\
\psi^{N_1-1} &in $\mathring{\tilde{Q}}_{N_1}$,\\
0 &\text{otherwise}.
\end{cases*}
\end{gather}
Now, we define our function $\phi_1$ by
\[ \phi_1  \coloneqq \sum_{k_2=1}^{N_2}\cdots \sum_{k_n=1}^{N_n} \phi^{k_2,\ldots,k_n}_1, \]
where we highlight that the supports of $\phi^{k_2^a,\ldots,k_n^a}_1$ and $\phi^{k_2^b,\ldots,k_n^b}_1$
have at most a trivial overlap for $(k_2^a,\ldots,k_n^a) \neq (k_2^b,\ldots,k_n^b)$ by construction.
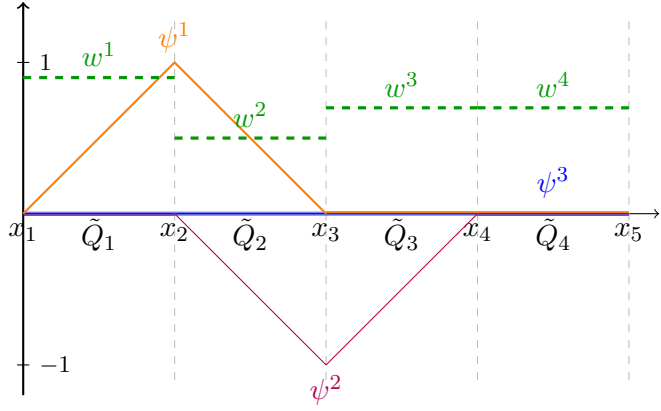
\begin{figure}
	\centering
	\begin{tikzpicture}[xscale=8, yscale=4]
	\def\eps{0.005}
	\draw[->] (0,0) -- (1.05,0) node[below right] {};
	
	\draw[thick,->] (0,-0.6) -- (0,0.7) node[left] {};
	
	\foreach \x/\label in {0.25/{$x_2$}, 0.5/{$x_3$}, 0.75/{$x_4$}, 1/{$x_5$}} {
		\draw[gray!50,dashed] (\x,-0.55) -- (\x,0.65);
		\node[below] at (\x,0) {\label};
	}
	\node[below] at (0,0) {$x_1$};
	
	\draw (-0.01,0.5) -- (0.01,0.5) node[right] {\footnotesize $1$};
	\draw (-0.01,-0.5) -- (0.01,-0.5) node[right] {\footnotesize $-1$};
	
	\node at (0.125,-0.07) {$\tilde{Q}_1$};
	\node at (0.375,-0.07) {$\tilde{Q}_2$};
	\node at (0.625,-0.07) {$\tilde{Q}_3$};
	\node at (0.875,-0.07) {$\tilde{Q}_4$};
		
	\draw[ultra thick,blue, opacity = 0.5] (0,0) -- (1,0);
	\node[blue,above] at (0.875,0.02) {$\psi^3$};

	\draw[thick,orange] (0,0) -- (0.25,0.5) -- (0.5,0);
	\draw[thick,orange, opacity = 0.9] (0.5,\eps) -- (1,\eps);
	\node[orange,above] at (0.25,0.5) {$\psi^1$};
	
	\draw[red!70!blue, opacity = 0.8] (0,-\eps) -- (0.25,-\eps);
	\draw[red!70!blue] (0.25,0) -- (0.5,-0.5) -- (0.75,0); 
	\draw[red!70!blue, opacity = 0.8] (0.75,- \eps) -- (1,-\eps); 
	\node[red!70!blue,below] at (0.5,-0.5) {$\psi^2$};

	\draw[very thick,dashed,green!60!black] (0,0.45) -- (0.25,0.45) node[midway,above] {$w^1$};
	\draw[very thick,dashed,green!60!black] (0.25,0.25) -- (0.5,0.25) node[midway,above] {$w^2$};
	\draw[very thick,dashed,green!60!black] (0.5,0.35) -- (0.75,0.35) node[midway,above] {$w^3$};
	\draw[very thick,dashed,green!60!black] (0.75,0.35) -- (1,0.35) node[midway,above] {$w^4$};
	
	\end{tikzpicture}
	\caption{Construction of $\psi^1, \ldots, \psi^3$ for given $w$ with partition
		$\tilde{Q}_1,\ldots,\tilde{Q}_4$.\protect\footnotemark}\label{fig : phi1}
\end{figure}
\footnotetext{The LLM GPT-4o mini was used during the creation of the tikz code for this figure.}
We construct $\phi_2,\ldots,\phi_n$ analogously to $\phi_1$ and subsequently define our candidate vector field as
$\phi \coloneqq (\phi_1, \phi_2, \phi_3)^T$. Next, we verify that $\phi \in \RT$ and the satisfaction of \eqref{eq:phi_eq_tv}. 

\textbf{Inclusion in $\RT$.} 
The divergence $\div \phi$ is constant per hyperrectangle $\tilde{Q}^{k_1,\ldots,k_n}$ because $\partial_j \phi_i = 0$ 
for $j \neq 0$ and $\partial_j \phi_j$ is piecewise constant as a derivative of a piecewise affine function for all
$j \in \{1,\ldots,n\}$. We now consider the facets of the (closed) hyperrectangles and their outer normal vectors
in order to show that the normal components of $\phi$ are continuous and constant per facet. 

Because the hyperrectangles $\tilde{Q}^{k_1,\ldots,k_n}$ are axis-aligned, the outer normal vectors are also axis-aligned
where defined. We give the proof for the normal components along the first coordinate for fixed $k_2,\ldots,k_n$
being fixed and use the abbreviated notation introduced above. Again, the claim for other coordinates follows
with the same argument. Let $E_m$ denote the common facet of the hyperrectangles $\tilde{Q}_m$ and $\tilde{Q}_{m+1}$ and 
$n_{E_m}$ denote the corresponding unit outer normal vector pointing into the positive direction.
Specifically, $n_{E_m} = e_1$, where $e_1$ is the first standard basis vector in $\R^n$. Consequently,
\begin{equation} \label{eq :norm_component}
\begin{aligned}
\phi \cdot n_{E_m} = \phi_1^{k_1,\ldots,k_n} = \sgn(w_m - w_{m+1}) \text{ on (the relative interior of) } E_m
\end{aligned}
\end{equation}
and thus $\phi \cdot n_{E_m}$ is constant per facet after restricting to their relative interiors,
Moreover, $\phi \cdot n_{E_m}$ is continuous for all $m$ because, after extending to the boundary,
$\phi_1$ is continuous on the tower $\overline{\bigcup_{m=1}^{N_1} {\tilde{Q}_m}}$ by construction.
As a consequence, we obtain that $\phi$ is a Raviart--Thomas vector field of lowest order. 

We now show that $\phi$'s normal trace vanishes at the boundary of $\Omega$. To this end, we denote by
$E_{0,k_2,\ldots,k_n} = \{a_1\} \times I^{k_2} \times \cdots \times I^{k_n}$ the facet of the hyperrectangle
$\tilde{Q}^{1,k_2,\ldots,k_n}$ that lies in $\partial \Omega$ and the $x_2,\ldots,x_n$-plane shifted
by $a_1$ along the first coordinate and by $E_{N_1,k_2,\ldots,k_n} = \{b_1\} \times I^{k_2} \times \cdots \times I^{k_n}$
the facet of the hyperrectangle $\tilde{Q}^{N_1,k_2,\ldots,k_n}$ that lies in $\partial \Omega$ and in the $x_2,\ldots,x_n$-plane
shifted by $b_1$ along the first coordinate.
We denote the corresponding outer unit normal vectors by $n_{E_{0,k_2,\ldots,k_n}}^- = -e_1$
and $n_{E_{N_1,k_2,\ldots,k_n}}^+ = e_1$. Then,
\begin{equation*}
\begin{aligned}
\phi \cdot n_{E_{0,k_2,\ldots,k_n}}^-& = -\psi^1 &= 0
&\text{ on } E_{0,k_2,\ldots,k_n},\\
\phi \cdot n_{E_{N_1,k_2,\ldots,k_n}}^+& = \psi^{N_1-1} &= 0
&\text{ on } E_{N_1,k_2,\ldots,k_n}.\\
\end{aligned}
\end{equation*}
Again, the same arguments apply to the other \emph{towers} of hyperrectangles and along the other coordinates
so that $\phi \in \RT$ follows.

\textbf{Boundedness in $L^\infty$.}  Again, we use the abbreviated notation above and restrict to the first
coordinates for several of our arguments since the others are completely analogous.
For fixed $k_2,\ldots,k_n$, we observe that $\phi_1^{k_2,\ldots,k_n} \in \{-1,0,1\}$
holds on the relative interior of each facet $E_m$ according to \eqref{eq :norm_component}. Since, $\phi_1$ is continuous
on $\overline{\bigcup_{m=1}^{N_1} {\tilde{Q}_m}}$ after possibly extending to the boundary and affine between every pair
of adjacent facets $E_{m}$ and $E_{m+1}$, $m \in \{1,\ldots,N_1\}$, we deduce that $|\phi_1| \leq 1$ holds in $\tilde{Q}_m$ for all
$m \in \N$. Since this holds for all possible tuples $(k_2,\ldots,k_n)$, we obtain $|\phi_1| \leq 1$ a.e.\ in $\Omega$.
The same arguments imply $| \phi_1 | \leq 1 $ and $| \phi_2 | \leq 1 $. Again, the same arguments apply for the other
coordinates so that
\[ \|\phi \|_{L^\infty(\Omega, \R^n)}
 = \operatorname{ess}\sup\left\{ \sqrt{\sum_{i=1}^n|\phi_i(x)|^2} : x \in \Omega \right\}. 
\]

\textbf{Realization of $\TV(w)$.} It remains to show that $\TV(w) = \int_\Omega \div\phi w$ holds.
Again, we use the abbreviated notation above and restrict to the first coordinates for several of our arguments
since the others are completely analogous. Due to the construction of $\phi$, we have
\[
\begin{aligned}
\int_{\Omega} \div \phi w 
&= \sum_{i=1}^{n} \int_{\Omega} \partial_i \phi_i w \\
&= \sum_{k_2=1}^{N_2}\cdots\sum_{k_n=1}^{N_n} \int_{\Omega} \partial_1 \phi_1^{k_2,\ldots,k_n} w 
   + \ldots + 
   \sum_{k_1=1}^{N_1}\cdots\sum_{k_{n-1}=1}^{N_{n-1}} \int_{\Omega} \partial_{n} \phi_n^{k_1,\ldots,k_{n-1}} w.
\end{aligned}
\]
We consider $\phi_1^{k_2,\ldots,k_n} = \sum_{m=1}^{N_1 - 1} \psi^m $ for arbitrary but fixed
$k_2,\ldots,k_n$. Because the support of $\psi^{m}$ is $\overline{\tilde{Q}_{m} \cup \tilde{Q}_{m+1}}$,
it holds
\begin{equation*}
\begin{aligned}
\int_{\Omega} \partial_1 \psi^{m} w  
&= \int_{\tilde{Q}_m}     \partial_1 \psi^{m} w 
 + \int_{\tilde{Q}_{m+1}} \partial_1 \psi^{m} w \\
&= \int_{\tilde{Q}_{m}} \ell_m^{-1} \sgn(w_{m} - w_{m+1}) w - \int_{\tilde{Q}_{m+1}} \ell_{m+1}^{-1}\sgn(w_{m} - w_{m+1}) w
&&{\scriptstyle{\eqref{eq:phim_der}}}\\
& = \frac{\meas \tilde{Q}_m}{\ell_m} \sgn(w_{m} - w_{m+1}) w_m - \frac{\meas \tilde{Q}_{m+1}}{\ell_{m+1}}
\sgn(w_{m} - w_{m+1}) w_{m+1} &&{\scriptstyle{w|_{\tilde{Q}_{m}}\equiv w_m}}.
\end{aligned}
\end{equation*}
Since the $\tilde{Q}_m$ are hyperrectangles separated by the interfaces $E_m$ along the first
coordinate, we have 
\[ \Ha^{n-1}(E_m) = \frac{\meas \tilde{Q}_m}{\ell_m} = \frac{\meas \tilde{Q}_{m+1}}{\ell_{m+1}} \]
for all $m \in \{1,\ldots,N_{1} - 1\}$. In turn, we obtain
\[ \int_{\Omega} \partial_1 \psi^{m} w = \Ha^{n-1}(E_m) \sgn(w_m - w_{m+1})(w_m - w_{m+1}) = \Ha^{n-1}(E_m)|w_m - w_{m+1}|.
\]
Since this argument holds for all $m\in\{1,\ldots,N_1-1\}$, we obtain
\[\int_{\Omega} \partial \phi_1^{k_2,\ldots,k_n} w =  \sum_{m = 1}^{N_1-1} \Ha^{n-1}(E_m)|w_{m} - w_{m+1}|. \]
The right-hand side accumulates the differences of the level sets of $w$ between neighboring hyperrectangles
$\tilde{Q}_{m}$ times the length of the interface between neighboring hyperrectangles along the first coordinate.

We extend the right-hand side to a sum between every pair $m_1$, $m_2 \in \{1,\ldots,N_1\}$ using productive zeros that
serve as jump heights between the non-neighboring hyperrectangles. Because the double sum counts each jump of $w$ along
the first coordinate twice, we correct it by the factor $1/2$ and obtain
\[
\int_{\Omega} \partial \phi_1^{k_2,\ldots,k_n} w 
= \frac{1}{2} \sum_{m_1=1}^{N_1} \sum_{m_2=1}^{N_1} \Ha^{n-1}(E_{\min\{m_1,m_2\}})
	|w_{\tilde{Q}_{m_1}} - w_{\tilde{Q}_{m_2}}|
	a^1(\tilde{Q}_{m_1},\tilde{Q}_{m_2})
\]
where $a^i(\tilde{Q}^\alpha,\tilde{Q}^\beta) = 1$ if $\tilde{Q}^\alpha, \tilde{Q}^\beta \in \tilde{\calQ}_{\tau}$ are adjacent along the $i$-th coordinate and $a^i(\tilde{Q}^\alpha,\tilde{Q}^\beta) = 0$ else.
\begin{multline*}
\int_{\Omega} \partial_1\phi_1 w = 
\sum_{k_2=1}^{N_2} \cdots \sum_{k_n=1}^{N_n}
\int_{\Omega} \partial \phi_1^{k_2,\ldots,k_n} w 
= \\ \frac{1}{2}
               \sum_{m_1=1}^{N_1} \sum_{m_2=1}^{N_1}  \sum_{k_2=1}^{N_2} \cdots \sum_{k_n=1}^{N_n}
               \Ha^{n-1}(E_{\min\{m_1,m_2\}})|w_{\tilde{Q}^{m_1,k_2,\ldots,k_n}}-  w_{\tilde{Q}^{m_2,k_2,\ldots,k_n}}|
	a^1(\tilde{Q}^{m_1,k_2,\ldots,k_n},\tilde{Q}^{m_2,k_2,\ldots,k_n}).
\end{multline*}
Extending this sum by further zeros and using that each $E_m$ is the intersection of the closures
of the two incident cubes along the first coordinate except for a $\Ha^{n-1}$-negligible set, we obtain
\[ \int_{\Omega} \partial_1\phi_1 w =
\frac{1}{2}\sum_{\tilde{Q}^i \in \calQ}\sum_{\tilde{Q}^j \in \calQ}
\Ha^{n-1}\big(\overline{\tilde{Q}^i}\cap\overline{\tilde{Q}^j}\big)|w_{\tilde{Q}^i} - w_{\tilde{Q}^j}|
a^1(\tilde{Q}^i,\tilde{Q}^j)
\]
and analogously for $\partial_2\phi_2,\ldots,\partial_n\phi_n$. Since at most one of
$a^1(\tilde{Q}^i,\tilde{Q}^j),\ldots,a^n(\tilde{Q}^i,\tilde{Q}^j)$ is non-zero for all
$\tilde{Q}^i,\tilde{Q}^j \in \calQ$ and in case all of them are zero we have
$\Ha^{n-1}(\overline{\tilde{Q}^i}\cap\overline{\tilde{Q}^j}) = 0$, we obtain
by summing from $1$ to $n$ that
\[ \int_{\Omega} \div \phi w =
\frac{1}{2}\sum_{\tilde{Q}^i \in \calQ}\sum_{\tilde{Q}^j \in \calQ}
\Ha^{n-1}\big(\overline{\tilde{Q}^i}\cap\overline{\tilde{Q}^j}\big)|w_{\tilde{Q}^i} - w_{\tilde{Q}^j}|
= \TV(w),
\]
where the second identity follows, for example, from Lemma 2.1 in \cite{manns2023on}\footnote{Lemma 2.1 in
\cite{manns2023on} assumes integer-valued functions. However, its proof requires only that the functions
are piecewise constant.}
\end{proof}

\begin{remark}
While we are not aware of quadrilateral Raviart--Thomas finite-element theory or applications thereof for
$n \ge 4$, our provided argument is generic and translates to higher dimensions if quadrilateral Raviart--Thomas
finite-element spaces are defined analogously to $n \le 3$.
\end{remark}

\begin{remark}
We believe that the results can be transferred to discretizations of bounded Lipschitz domains if, for example, $\phi$ is set to zero on quadrilaterals that are cut by
$\partial \Omega$ within the supremization 
in \eqref{eq:dfn_tv_tau}.
\end{remark}

As a corollary, we close an open question from
\cite{schiemann2025discretization} what an
asymptotically sharp (smallest possible)
constant for relationship between $\TV(w_\tau)$
and $\TV(w)$ is if $w_\tau$, $w$ are integer-valued and
the $w_\tau$ defined on rectangular partitions.

\begin{corollary}[Tight variant of Theorem 3.12 and Theorem 3.13 in \cite{schiemann2025discretization} for $n \in \{1,2,3\}$]
Let $w \in \BV(\Omega) \cap \{ w \in L^1(\Omega) : w(x) \in \{w_1,\ldots,w_M\} \text{ for a.e.\ } x\in\Omega \}$, $w_i \in \mathbb{Z}$
for all $i \in \{1,\ldots,M\}$, $M \in \N$.
Then, there exists a sequence $(w_\tau)_\tau \subset \BV(\Omega) \cap \{ w \in L^1(\Omega) : w(x) \in \{w_1,\ldots,w_M\} \}$ such that
\[ \limsup_{\tau\searrow 0} \TV(w_\tau) \le \sqrt{n}\TV(w). \]
\end{corollary}
\begin{proof}
The sequence $(w_\tau)_\tau$ is defined as $w_\tau \coloneqq R^{\{w_1,\ldots,w_M\}}_{\sigma_\tau}(w)$, where the operator
$R$ is defined as in Definition 2.11 in \cite{schiemann2025discretization}, the $h$-grid in \cite{schiemann2025discretization}
corresponds to our $\tau$-grids, a (third) sequence of $\sigma_\tau$-grids (corresponding to $\tau_h$-grids in
\cite{schiemann2025discretization}) is embedded into the $\tau$-grids so that Assumption 2.3 in \cite{schiemann2025discretization}
is satisfied, and we have $\sigma_\tau \tau^{-1} \searrow 0$ as $\tau\searrow 0$. The key observation is that the function $w_\tau$
is piecewise constant per cell of the $\tau$-grid by definition so that the assumption of \Cref{lem:sqrtd_estimate_for_dg0}
is satisfied and we obtain, incombination with the second $\limsup$-inequality in Theorem 3.13 in \cite{schiemann2025discretization}, that
\[ \limsup_{\tau\searrow 0} \TV(w_\tau)
   \le \limsup_{\tau\searrow 0} \sqrt{n}\TVt(w_\tau)
   \le \sqrt{n}\TV(w).
\]
\end{proof}

\begin{theorem}
$\sup_{\tau\searrow 0} \TV_\tau(w_\tau) \le C$ 
and $\sup_{\tau\searrow 0} \|w_\tau\|_{L^1(\Omega)} \le C$ for some $C > 0$ with $w_\tau = P_{\tau} w_\tau$ implies
that there exist a subsequence $\{w_\tau\}_{\tau}$ and $w \in \BV(\Omega)$ such that
$w_\tau \to w$ in $L^1(\Omega)$
\end{theorem}
\begin{proof}
	\Cref{lem:sqrtd_estimate_for_dg0} implies that $\sup_{\tau\searrow 0} \TV(w_\tau) \le \sqrt{n}C$ holds.
	In combination with $\sup_{\tau\searrow 0} \|w_\tau\|_{L^1(\Omega)} \le C$, Theorem 3.23 from
	\cite{ambrosio2000} yields the claim.
\end{proof}

\section{Computational Experiments}\label{sec : comp_results}
We consider a model example from the existing literature (see \cite[\S5]{manns}) which reads
\begin{gather}\label{eq:ocpc}
    \begin{aligned}
\min_{u,w} \ &\frac{1}{2}\|u - u_d \|^2_{L^2(\Omega)} + \alpha \TV(w) \\
\text{s.t. } & 
    -\varepsilon \Delta u + c_1 \cdot \nabla u + c_2 u w = f \quad\text{in } \Omega \\
    &u= 0 \quad\text{on } \{0,1\} \times (0,1) \cup ((0,0.25) \cup (0.75,1)) \times \{1\} \\
    &u= \sin(2 \pi (x_1 - 0.25)) \quad\text{on } (0.25,0.75) \times \{1\} \\
    &\partial_n u = 0 \quad\text{on } (0,1) \times \{0\},
    \end{aligned}\tag{OCP$_C$}
    \end{gather}
    where $\alpha = 10^{-5}, \varepsilon = 0.04$, $c_2 = 4$, $c_1(x) = (\begin{matrix} \sin(\pi x_1)
    & \cos(2 \pi x_2)\end{matrix})^T$ for $x \in \Omega$, 
    $f(x) = \sin(2 \pi x_1 + 2 \pi x_2) + 3$ for $x \in \Omega$.
While it may be considered to be slightly more complicated than our model problems \eqref{eq : ocpe} and \eqref{eq : ocpc}, it allows for a direct comparison with the results from \cite[\S5]{manns}.
The verification of the assumptions for \eqref{eq:ocpc} can be established 
using the same strategies that we have used for \eqref{eq : ocpe} and
\eqref{eq : ocpc} or, alternatively, the non-standard elliptic regularity from 
\cite{groger1989aw} that gives $W^{1,p}$-regularity of the state for $p > 2$
and thus $L^\infty$-bounds since the domain is in 2D.

\paragraph{Key Insight for Acceleration of OBBT}

The key driver for both the quality / tightness and the compute time for
McCormick relaxations is the OBBT procedure. In \cite{manns}, the very expensive
OBBT procedure was applied with relatively coarse grids (coinciding for local
averaging and controls) with up to $48 \times 48$ 
square cells. In this work, we push this size further to grids of sizes
$128 \times 128$. In order to do so, we have integrated several acceleration 
techniques into the OBBT Algorithm 1 from \cite{manns}. We briefly note that,
in addition to the techniques described below, we have also assessed the
effect of replacing the state variable in the McCormick inequalities
and objective by the (discretized) solution operator of the PDE. This,
however, has led to a dense system of linear inequalities instead to
a sparse one, which has turned out to be much more compute-intensive
to solve. Using an outer approximation strategy that adds the inequalities only when
necessary, that is, violated, during the solution process has led to
the quick inclusion of a massive amount of inequalities and driven
up solution times even higher. We have therefore discontinued this approach.

Instead, our acceleration approach is based on the key observation that
when considering the minimizer / maximizer of the average of the state
variable in a specific cell, it is intuitive from the continuity properties
of the state variable that the same or a very similar state vector is the 
solution to the minimization / maximization for nearby grid cells.
Now, assuming a bound for a single cell was computed using primal
simplex to solve the corresponding LP of the form \eqref{eq : obbt_problems}, the simplex basis remains feasible when just the objective is changed,
which is the case if the constraints are not updated after each bound
optimization but only after the whole set of bounds were optimized one after 
another (one OBBT sweep). In addition, only four entries have to be changed
in the constraint set when the information of one bound update is added.

\paragraph{Experimental Procedure} \label{sec : experimental_procedure}
In our experiments, we follow the Ritz--Galerkin ansatz described in \cite[\S4.2]{manns} 
for the PDE in \eqref{eq:ocpc} and use the following discretization scheme
\begin{itemize}
\item for $u$: first-order Lagrange elements on $4 \times 128 \times 128$ triangles,
\item for $w$: piecewise constant functions on $128 \times 128 $ squares.
\end{itemize}
Regarding the total variation, we use the same discretization as in \cite{manns}, which
in turn is based on \cite{manns2025discrete}. While this discretization may give anisotropic
effects, this is not important towards our main goal---the investigation of acceleration techniques
for the OBBT Algorithm---but makes our setting comparable to \cite{manns}.
Since we aim to accelerate the OBBT Algorithm 1 from 
\cite{manns}, we conduct the computational experiments described in the
following sections. We conducted all of the experiments on a node of the Linux
HPC cluster LiDO3 with two AMD EPYC 7542 32-Core CPUs and 64 GB RAM. The computations were limited to 1 CPU.

\subsection{First Experiment for Solver Configuration Determination}
We use a fine square grid of $N_h$ for the McCormick relaxation, where we choose the mesh size / side length of the squares
as $\tau = h = 2^{-7}$ so that $N_h = 2^7 \times 2^7 = 2^{14}$. We run one OBBT
Algorithm sweep using two different solvers for the 
LPs  \eqref{eq : obbt_problems}, three different warm start variants,
and two cell ordering variants. Regarding the solvers, we use \texttt{Gurobi} \cite{gurobi} and \texttt{HiGHS} \cite{highs}. As warm start variants we compare
\begin{itemize}
\item \texttt{cold}\quad We solve each iteration with an interior point method and start each OBBT sweep with updated bounds from the
preceding sweep. No LP warm starting is used.
\item \texttt{semi-warm}\quad We solve the initial LP using an interior point method and then use primal simplex without updating the constraints.
We also start each OBBT sweep with updated bounds from the preceding sweep. Since we only update the objective between subsequent iterations of
the same OBBT sweep and use the optimal simplex basis from preceding iteration, the basis remains feasible for primal simplex.
\item \texttt{warm}\quad We solve the initial LP using an interior point method and then use dual simplex for subsequent iterations.
The constraint set is updated with the computed bound after each LP solve.
Here, we first optimize all of the lower bounds and then all of the upper bounds. As a consequence, we can expect that the bound improvement of the upper bounds is more pronounced than for the lower bounds, since more information is available in the warm start mode. 
\end{itemize}
In addition, we compare two different ways of ordering the grid cells, namely 
\begin{itemize}
\item \texttt{diagonal}\quad a bottom-up traversal of the grid cells using a typical ordering of grid cells for PDE computations (the default
of dolfinx \cite{DOLFINx} that is used in our computations),
\item \texttt{snake}\quad a row-wise bottom-up traversal of the grid cells with high spatial locality of the grid cells in the resulting serialization;
\end{itemize}
see \cref{fig:traversals} for a visualization of the two traversal variants.
\begin{figure}[ht]
	\centering
	\begin{subfigure}{0.48\textwidth}
		\centering
		\includegraphics[width=\linewidth]{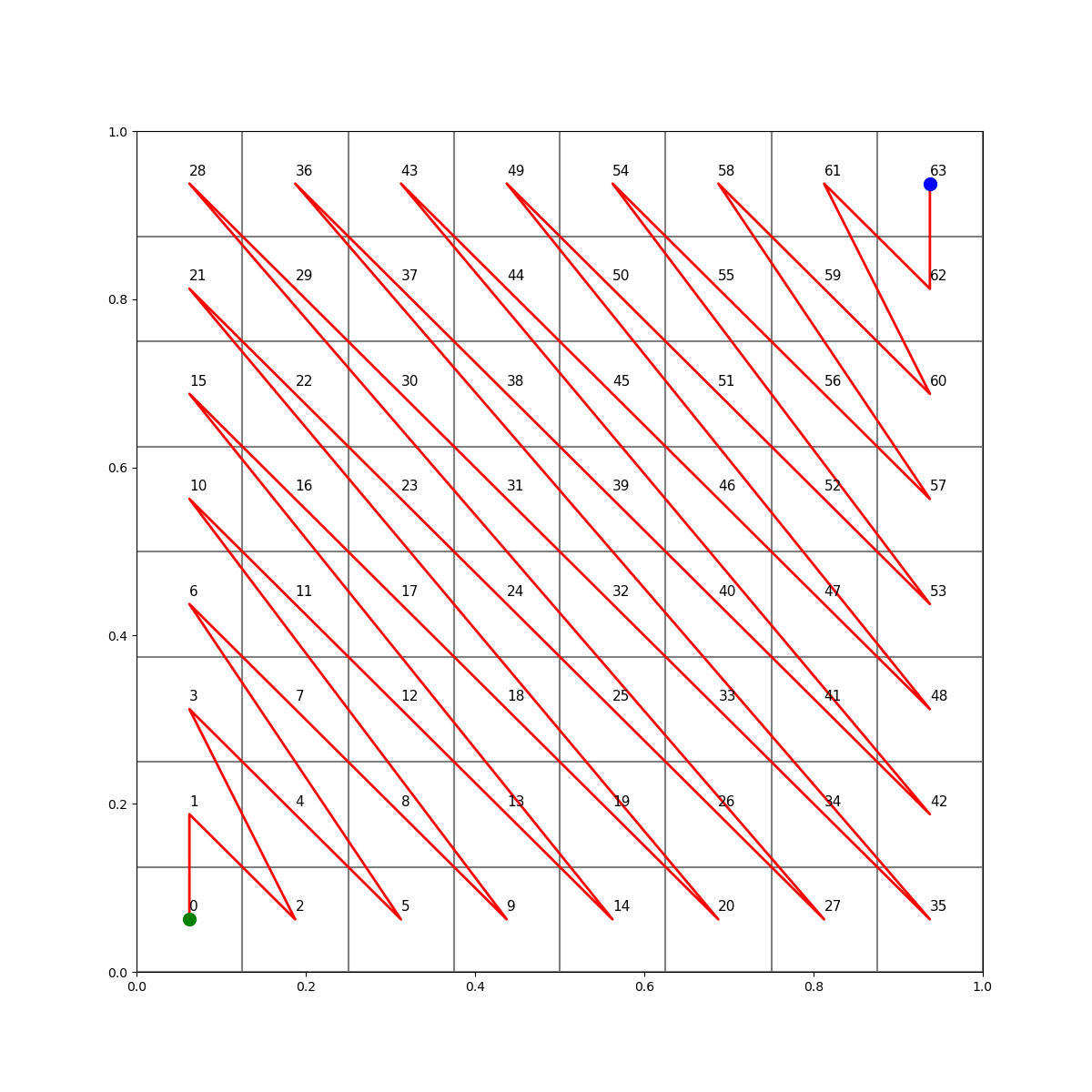}
		\caption{Diagonal up-bottom traversal}
		\label{fig:bild1}
	\end{subfigure}
	\hfill
	\begin{subfigure}{0.48\textwidth}
		\centering
		\includegraphics[width=\linewidth]{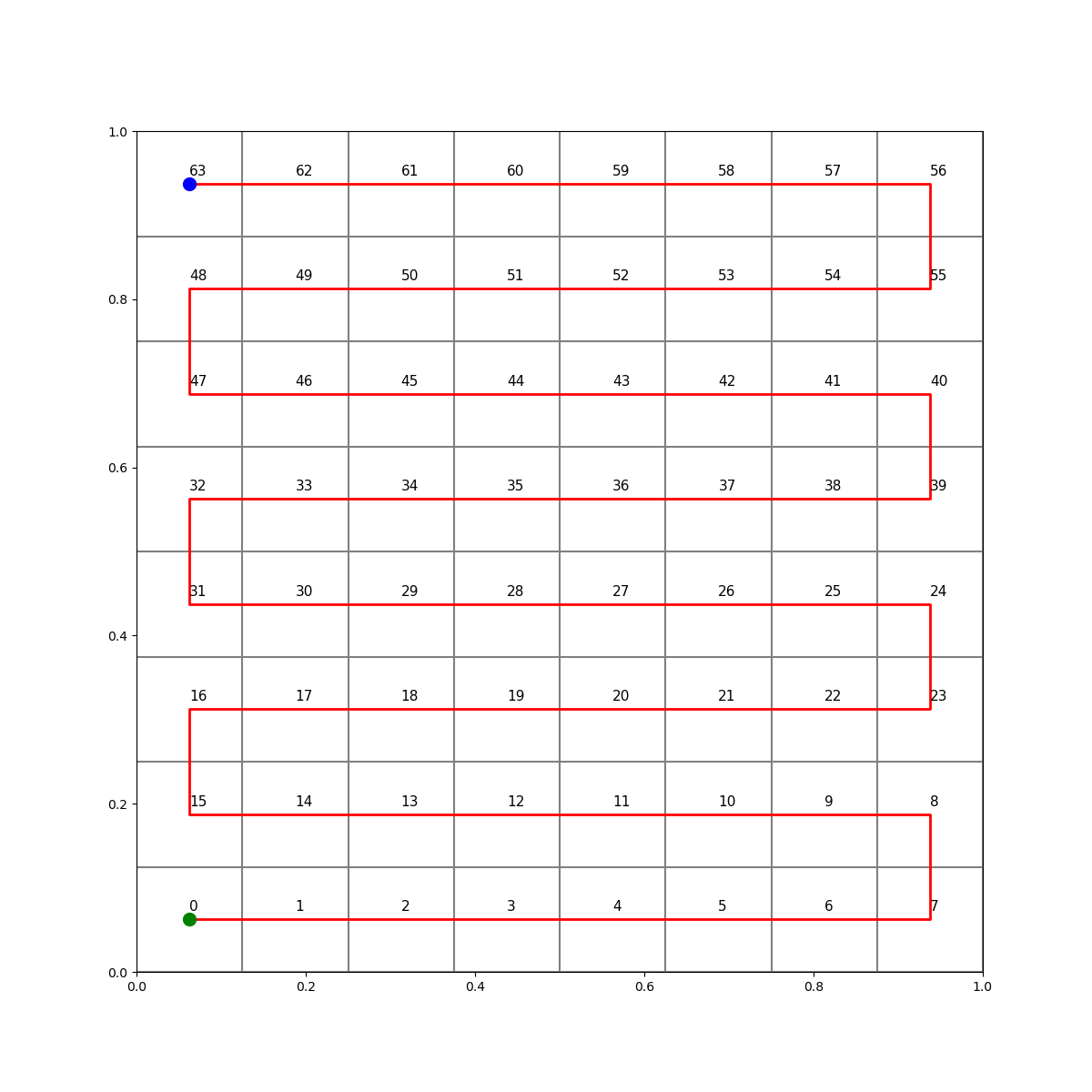}
		\caption{Row-wise snake traversal}
		\label{fig:bild2}
	\end{subfigure}
	\caption{Example of the assessed grid cell traversals on an $8 \times 8$ square grid.}
	\label{fig:traversals}
\end{figure}
Regarding initialization, we set all upper bounds to $10^3$ and all lower bounds to $-10^3$.

\paragraph{Results}
The runs for the configurations \texttt{HiGHS} with \texttt{cold}, \texttt{snake} 
and \texttt{Gurobi} with \texttt{cold}, \texttt{snake} were canceled after 48 hours
as just around 10\,\% and 3\,\% of the LP solves for the sweep were completed
at this time. 
Since the LP solves for the choices \texttt{snake} and \texttt{diagonal}
do not differ for \texttt{cold} (only their order differs), we omitted
running \texttt{HiGHS} with \texttt{cold}, \texttt{diagonal}
and \texttt{Gurobi} with \texttt{cold}, \texttt{diagonal}.
The runs for the configurations \texttt{Gurobi} with \texttt{warm}, 
\texttt{diagonal} and \texttt{Gurobi} with \texttt{warm}, \texttt{snake}
finished approximately within 12 and 8 hours so that the \texttt{snake}
traversal option was able to reduce the running time by about one third
in this case. The runs for the configurations \texttt{Gurobi} with
\texttt{semi-warm}, \texttt{diagonal} and \texttt{Gurobi} with \texttt{warm}, 
\texttt{diagonal} had comparable running times.
The runs for the configurations \texttt{HiGHS} with \texttt{warm}, 
\texttt{snake} were canceled after 48 hours as just around 25\,\%
of the LP solves for the sweep were completed as this time and thus
the performance was clearly inferior to \texttt{Gurobi} with \texttt{warm}, 
\texttt{diagonal} and \texttt{Gurobi} with \texttt{warm}, \texttt{snake}
already. Since \texttt{snake} did reliably outperform \texttt{diagonal},
the configuration \texttt{HiGHS} with \texttt{warm}, \texttt{diagonal}
was not run.
The runs for the configurations \texttt{HiGHS} with \texttt{semi-warm}, 
\texttt{diagonal} and \texttt{HiGHS} with \texttt{semi-warm}, \texttt{snake}
finished approximately within 2.33 and 2.18 hours so that the \texttt{snake}
traversal option was able to reduce the running time by about 6\,\%
in this case.
The recorded running times are tabulated in \cref{tab:pretest1} below.

Because the constraint set is updated after every LP solve, the bounds obtained
with one sweep in a \texttt{warm} configuration must be superior to the bounds 
computed with a \texttt{semi-warm} configuration when all other settings are
identical. As a consequence \texttt{semi-warm} does not seem to be a promising 
approach when used in combination with \texttt{Gurobi}. Indeed, the bounds
obtained with a \texttt{warm} sweep are much tighter (41\,\% on average 
after this one sweep) as one can observe in  \cref{fig:comparison} further down below. By contrast, the 4-5 times lower running 
times of the configuration \texttt{HiGHS} with \texttt{semi-warm} allow for a
much higher number of sweeps to be executed within the same time period
so that it is not immediate if \texttt{Gurobi} with \texttt{warm}
or \texttt{HiGHS} with \texttt{semi-warm} is preferable.

\begin{table}[ht]
	\centering
	\caption{Running times (in seconds) for one sweep of the OBBT Algorithm
	using different warm start configurations of Gurobi and HiGHS and different
	grid traversals on an $128 \times 128$ grid.}
	\label{tab:pretest1}	
	\begin{threeparttable}[t]
		\centering
		\begin{tabular}{r|cccccc}
			\toprule
			 & \multicolumn{3}{c}{\texttt{Gurobi}} & \multicolumn{3}{c}{\texttt{HiGHS}} \\
			\cmidrule(r){2-4} \cmidrule(l){5-7}
			& \texttt{cold} & \texttt{semi-warm} & \texttt{warm} 
			& \texttt{cold} & \texttt{semi-warm} & \texttt{warm} \\
			\midrule
			snake 
			& $^{*}$ 
			& $2.85 \times 10^4$ 
			& $2.90\times 10^4$ 
			& $^{**}$ 
			& $7.88 \times 10^3$ 
			& $^{***}$ \\
			diagonal 
			& 
			& $4.47 \times 10^4$ 
			& $4.33 \times 10^4$  
			& 
			& $8.42 \times 10^3$ &  \\
			\bottomrule 
			 \end{tabular}
		\begin{tablenotes}
		\item[$^{*}$] canceled after $1.73 \times 10^5$\,s;
		3\,121 of 32\,768 LPs solved.
		\item[$^{**}$] canceled after $1.73\times 10^5$\,s; 194 of 32\,768 LPs
		solved.
		\item[$^{***}$] canceled after $1.73\times 10^5$\,s; 8\,304 of 32\,768 LPs solved.
		\end{tablenotes}
\end{threeparttable}
\end{table}

\paragraph{Outcome} Due to their performance in this experiment, we choose \texttt{HiGHS} with \texttt{semi-warm},
\texttt{snake} and \texttt{Gurobi} with \texttt{warm}, \texttt{snake} as the configurations for the next experiment.

\subsection{Second Experiment for Solver Configuration Determination}
We run the OBBT Algorithm 1 in \cite{manns} for 72 hours or until convergence for both \texttt{HiGHS} with \texttt{semi-warm}, \texttt{snake} and
\texttt{Gurobi} with \texttt{warm}, \texttt{snake} on the same grid. We use the same initialization as before with all upper bounds set to $10^3$
and all lower bounds set to $-10^3$. In order to avoid the numerical problems described in \cite[\S4.3]{manns}, we add an offset of $10^{-4}$ to
updates of upper bounds and subtract the same offset to updates of lower bounds. We determine OBBT Algorithm 1 as (numerically) converged
once the maximum difference to the previous bound update is below $10^{-2}$.

In addition, \texttt{HiGHS} required additional numerical stabilization to be able to reliably solve the first iteration of a sweep using
the interior point method. In particular, we did not improve corresponding upper and lower bounds $u_\ell^{\hat{i}}$ or $u_u^{\hat{i}}$
further if their difference $u_u^{\hat{i}} - u_\ell^{\hat{i}}$ fell below the termination tolerance of $10^{-2}$.
In addition, we did not accept the bound update when its absolute value fell below $10^{-4}$.

\paragraph{Results}
Within the prescribed time limit of 72 hours, \texttt{Gurobi} with \texttt{warm},
\texttt{snake} was able to complete four sweeps which took almost 61 hours and the time limit was reached during the fifth. In particular, 
the convergence criterion was not satisfied at this point.
By contrast, \texttt{HiGHS} with \texttt{semi-warm}, \texttt{snake} 
did converge in 13 sweeps just short of 23 hours.
For a cell in the center of the domain, the increase / decrease of
the bounds over time is visualized in \cref{fig:bound_over_sweeps}.
\begin{figure}[ht]
\centering
\begin{subfigure}{0.48\textwidth}
\centering
\includegraphics[width=\linewidth]{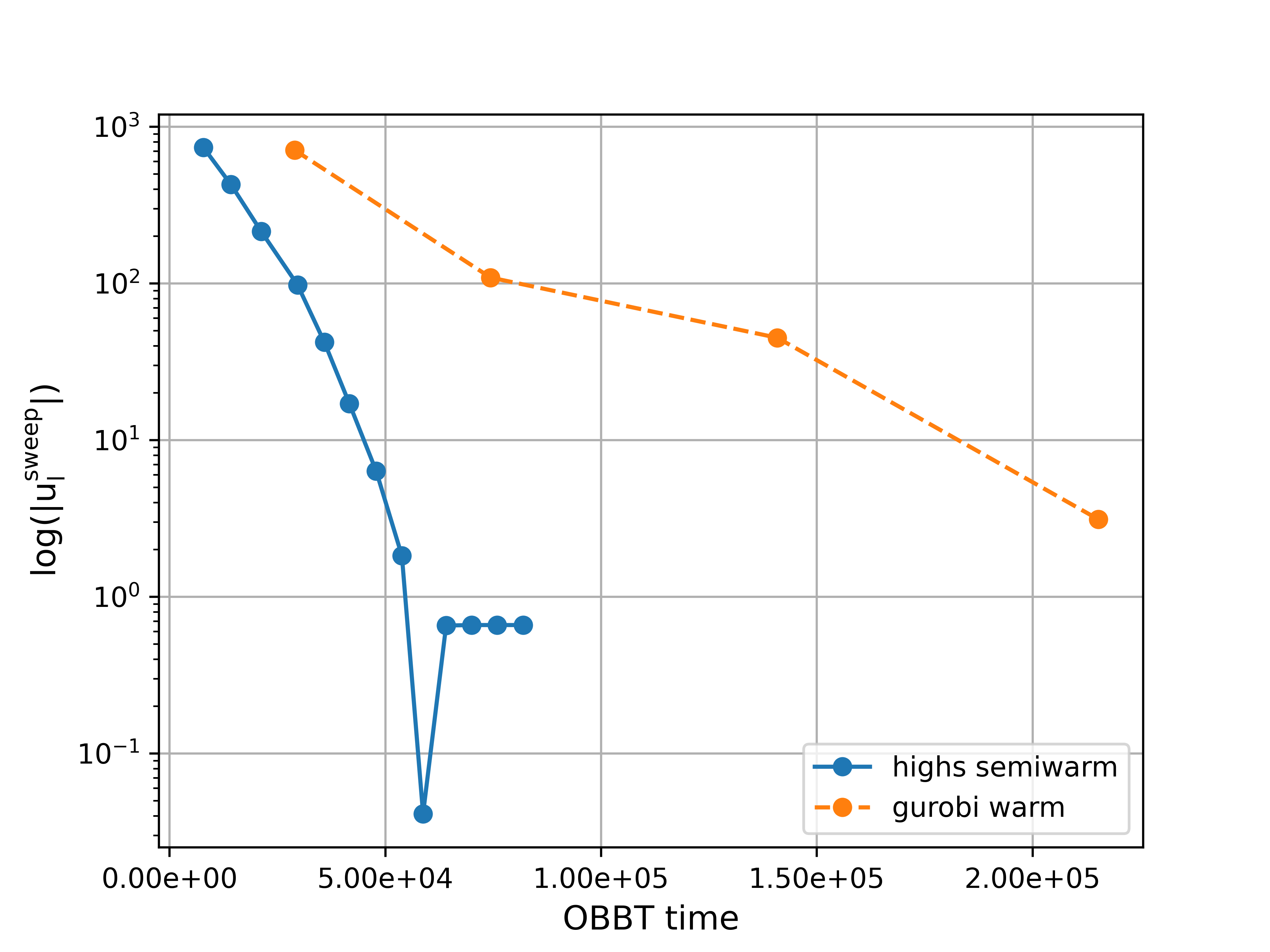}
\caption{Log-scaled value of $|u_\ell^{8319}|$ over time.\\
Note that $u^{8319}_\ell$ flips its sign in sweep 9.}
\label{fig:ul_over_sweeps}
\end{subfigure}
\hfill  
\begin{subfigure}{0.48\textwidth}
\centering
\includegraphics[width=\linewidth]{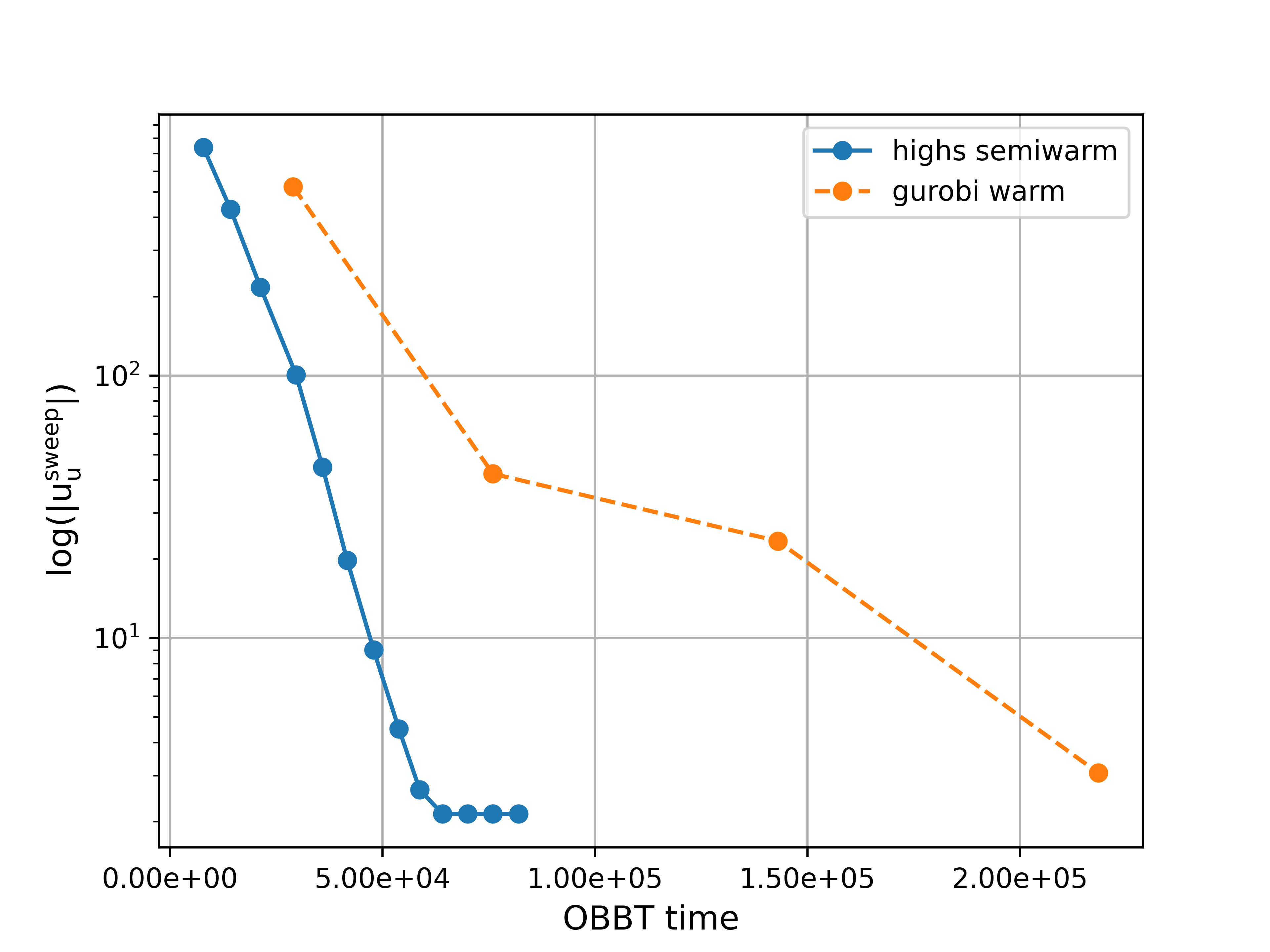}
\caption{Log-scaled value of $|u_u^{8319}|$ over time.\\~}
\label{fig:uu_over_sweeps}
\end{subfigure}
\caption{Lower (left) and upper bound (right) bound for grid cell 8319
over the course of the OBBT Algorithm running time (in seconds) for a square $128 \times 128$ mesh of the size using \texttt{Gurobi} with \texttt{warm},
\texttt{snake} (orange) and \texttt{HiGHS} with \texttt{semi-warm},
\texttt{snake}, where the bullet points mark the time of a completed
sweep.}\label{fig:bound_over_sweeps}.
\end{figure}
All upper bounds after after one and four sweeps of \texttt{Gurobi} with \texttt{warm}, \texttt{snake} and after one and 13 sweeps of \texttt{HiGHS}
with \texttt{semi-warm}, \texttt{snake} are visualized in a heatmap in
\cref{fig:comparison}.
\begin{figure}[h]
	\centering
	\includegraphics[width=0.75\linewidth]{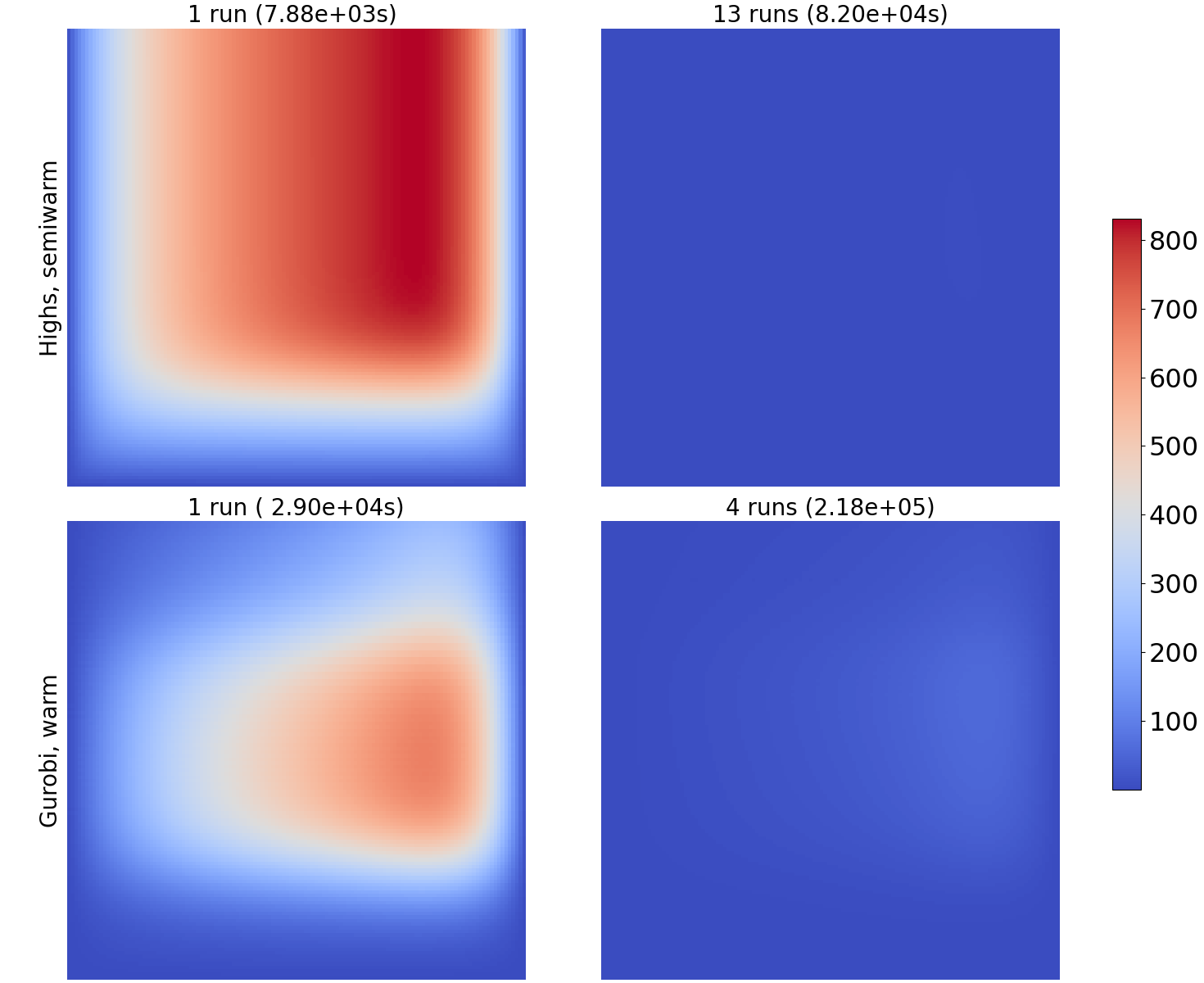}
	\caption{Value of $u_u^i$ on the mesh squares $Q_h^i$ after running 1 and 13 sweeps of \texttt{HiGHS} with \texttt{semi-warm}, \texttt{snake} (bottom row) and after running 1 and 4 sweeps of Gurobi with the warm start (top row).}
	\label{fig:comparison}
\end{figure}

\paragraph{Outcome} Due to the superior performance in this experiment, we choose \texttt{HiGHS} with \texttt{semi-warm},
\texttt{snake} as the configuration for the next experiment.

\subsection{Qualitative Assessment of the Bounds}
We execute the OBBT Algorithm using \texttt{HiGHS} with \texttt{semi-warm}, \texttt{snake} for $h = \tau$ and the choices $h = 2^{-2}$, $\ldots$, $2^{-7}$
until convergence, which yields variable bounds $u_\ell^i$, $u_u^i$. Then, we solve the convex problems \eqref{eq : mcchh} by means of \texttt{Gurobi}
using the sharpest variable bounds obtained with the final sweep OBBT Algorithm for all of these values of $h$, which thus yields approximate lower
bounds for \eqref{eq:ocpc}. We note that we have observed numerical problems and suboptimal terminations of \texttt{Gurobi} when solving \eqref{eq : mcchh},
in particular for low values of $h$ in the default setting. We were able to overcome these issues by setting the parameters $\operatorname{ScaleFlag} = 2$, $\operatorname{BarHomogeneous} = 1$, and $\operatorname{NumericFocus} = 2$. To compare the obtained approximate lower bounds to an upper bound, we compute an
upper bound for \eqref{eq:ocpc} by executing a local gradient-based NLP solver on it. Due to the
small gap between the upper bound computed with the NLP-solver and the lower bound computed on the baseline discretization in \cite{manns}, we expect
the solution to be close to a global minimum. However, our new ability to scale the the OBBT Algorithm to finer meshes comes at the cost that
we were not able to execute the OBBT Algorithm to convergence on a baseline discretization as in \cite{manns}. In this case, LPs with a substantial
computational demand were generated during the OBBT Algorithm, thereby making the overall algorithm execution intractable since thousands of LPs have to
be solved and we had to cancel it after 72 hours with zero completed sweeps. We have not figured out why this is the case yet and have observed
this before and after transferring our acceleration techniques to the baseline discretization. Due to the finer meshes we allow now and the fact that we
can run the OBBT Algorithm to convergence, we expect the solutions to \eqref{eq : mcchh} be close to true lower bounds for the smaller 
values of $h$.

\paragraph{Results}
We observe that the execution time of the OBBT Algorithm followed a monotonically increasing trend when $h$ is decreased, ranging
from $2.6558 \times 10^{3}$\,s for $h = 2^{-2}$ to $8.2032 \times 10^{4}$\,s for $h = 2^{-7}$. Notably, for all values of
$h$ except $h = 2^{-2}$, the OBBT Algorithm did converge after 13 sweeps, indicating a very similar bound tightening behavior that seems to be independent of the mesh size. For $h = 2^{-2}$, the OBBT Algorithm required only 12 sweeps. The approximate lower bounds on
\eqref{eq:ocpc} that are achieved by solving \eqref{eq : mcchh} afterwards also follow a monotonically decreasing
trend and seems to converge to (approximately) $3.2293$ with the relative difference between the approximate lower bounds for $h = 2^{-4}$,
$h = 2^{-5}$, $h = 2^{-6}$ and $3.2293$ (the value for $h = 2^{-7}$) being below or slightly above to $10^{-4}$, which is a typical tolerance for relative
duality gaps in many solvers. The compute times for solving \eqref{eq : mcchh} lay between $2.0920\times 10^{1}$\,s for $h = 2^{-5}$
and $5.2660\times 10^{1}$\,s for $h = 2^{-7}$. The execution of the NLP-based upper bound computation yielded an objective
value of $3.2294$, which is slightly above the finest approximate lower bound. It took $4.8000 \times 10^{1}$\,s to execute this computation.
For sake of completeness, we note that without executing the OBBT Algorithm, the initial variable bounds prescribed on \eqref{eq : mcchh}
yield an objective value that is numerically zero, which is a trivial bound on \eqref{eq : ocpc}. The objective values and compute times
are tabulated in \cref{tbl:overall_assessment_results}. Since one can also compute approximate lower bounds with intermediate results
of the OBBT Algorithm, we assess how the objective value of \eqref{eq : mcchh} evolves over the different sweeps of the OBBT Algorithm for $h = 2^{-2}, \ldots, 2^{-7}$.
 For all of our mesh sizes, the induced lower bound settles after approximately 9 sweeps and barely increases afterwards. We have visualized this in \cref{fig:approximate_bound_over_sweep}.
\begin{table}[h]
	\caption{Running times (in seconds) of the OBBT Algorithm for different mesh sizes as well as running times (in seconds) and achieved
	approximate lower bounds of the solution of \eqref{eq : mcchh} with the tightened variable bounds in comparison
	to the NLP-based upper bound and its compute time.}\label{tbl:overall_assessment_results}
	\centering
	\begin{adjustbox}{width=\textwidth}	
		\begin{tabular}{r|ccccccc}
			\toprule
			& \multicolumn{6}{c}{Approximate Lower Bound by Solving \eqref{eq : mcchh}} & NLP-based Upper Bound
			\\
			\cmidrule(lr){2-7} \cmidrule(lr){8-8}
			\\
			$h = \tau = $ & $2^{-2}$& $2^{-3}$ &$2^{-4}$ &  $2^{-5}$ & $2^{-6}$  & $2^{-7}$ & \\
			\midrule
			Time OBBT [s] & \num{2.6558e+03}   &\num{4.9871e+03} &\num{8.1295e+03} &\num{1.7320e+04} &\num{1.7873e+04}  &\num{8.2032e+04}
			\\
			Time for Solving [s] & \num{2.4910e+01}& \num{2.3800e+01}& \num{2.8250e+01}& \num{2.0920e+01}& \num{2.8910e+01}& \num{5.2660e+01} & \num{4.8000e+01}
			\\
			Objective value &\num{3.2389e+00}   &\num{3.2328e+00} &\num{3.2299e+00} &\num{3.2295e+00} &\num{3.2294e+00} &\num{3.2293e+00}
			& \num{3.2294} \\        
			\bottomrule
		\end{tabular}
	\end{adjustbox}
\end{table}

\begin{figure}[h]
	\centering
	\includegraphics[width=\linewidth]{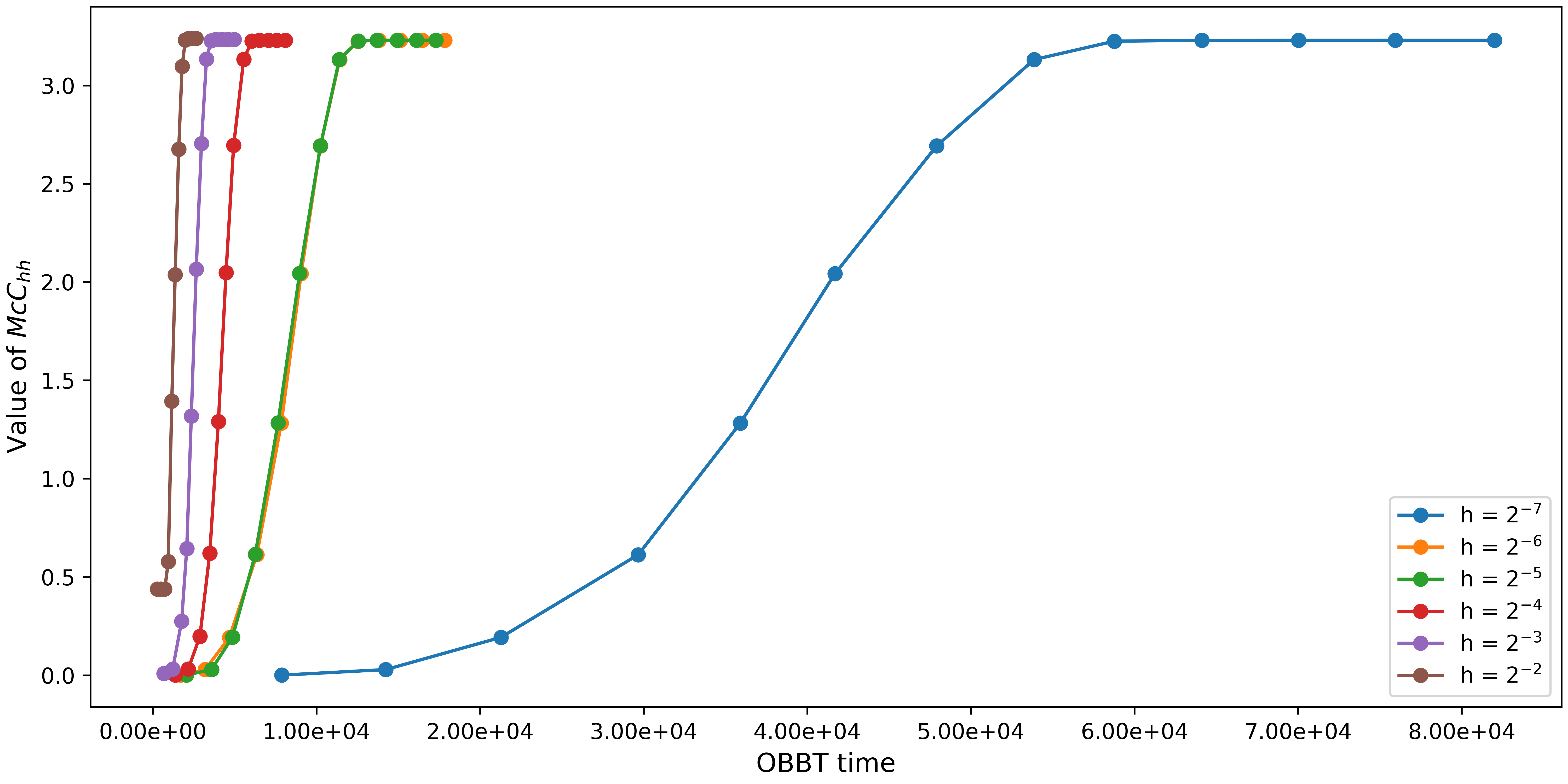}
\caption{Development of the value of $\eqref{eq : mcchh}$ over the sweeps of the OBBT Algorithm for different mesh sizes.}	\label{fig:approximate_bound_over_sweep}
\end{figure}

\paragraph{Interpretation}
The results show that the NLP-based upper bound is indeed (close to) a global minimum. In addition, the compute times for executing
an NLP solver on \eqref{eq:ocpc} or solving \eqref{eq : mcchh} after variable bounds are determined is negligible compared
to the bound tightening procedure. The bound tightening procedure on the other hand already gives high quality bounds even for
relatively large values of $h$. For our example, if one is satisfied with a relative accuracy of $2 \times 10^{-4}$ for the approximate
lower bound, it would have been possible to use the relatively coarse value of $h = 2^{-4}$ for \eqref{eq : mcchh}, thereby saving
an order of magnitude in compute time compared to $h = 2^{-7}$. 
The compute time can be also saved by posing a less strict termination criteria for the OBBT Algorithm since the later OBBT sweeps do not improve the resulting approximate lower bound by much. For our example, there is no numerical difference in the optimal value of \eqref{eq : mcchh} between the penultimate and the last sweeps for all the mesh sizes and the relative difference between the optimal value of \eqref{eq : mcchh} for is below or slightly above $6\times 10^{-5} $ for the sweeps $10$ to $13$ for all $h = 2^{-3}, \ldots, 2^{-7}$. Earlier
termination after 9 sweeps allows to save up to $19\%$ of the compute time for OBBT.

\section{Conclusion}\label{sec : concluson}
We have successfully demonstrated that the assumptions underlying the derivation of McCormick
relaxations and their approximations in \cite{manns} can be verified for PDEs, in particular
elliptic ones, on multi-dimensional domains.
On the computational side, we have been able to scale the approach to significantly finer meshes
underlying the approximate McCormick relaxations than what has been possible before. The driver
of the computational burden is the OBBT Algorithm, which we were able to accelerate significantly
by exploiting and combining the warm starting capabilities of the simplex algorithm and the
continuity properties of the PDE solutions.
Since the meshes we have used are still of academic size, further research is necessary
to scale the approach towards more realistic and in particular non-academic settings
and discretizations.

\bibliographystyle{plain}
\bibliography{references}

\appendix

\section{Appendix}
\begin{lemma} \label{thm : aux_estim_M}
    Let $(u,w)$ be feasible for \eqref{eq : ocpe}. Let the test function $\zeta_M$, the constant $M$, and the set $E(k)$ for $k >0$ be defined as in
    the proof of \cref{thm:ocpe_bound}, then the following estimation holds:
    \begin{equation*}
    \int_{A(k)}  {u}\zeta_M \leq M \int_{A(k)}  {u} {u}.
    \end{equation*}
\end{lemma}
\begin{proof}
    Firstly, we split the integral into the sum of two terms depending on whether ${u} \geq k$ or ${u} \leq -k$ and, then, use the properties of the integral and the fact that $A(k) = \{x \in \Omega | u(x) \geq k\} \cup \{x \in \Omega | u(x) \leq -k\}$:
    \begin{equation*}
        \begin{aligned}
        \int_{A(k)}  {u}\zeta_M &= M \biggr[ \int_{ {u}\geq k}  {u}( {u}-k) + \int_{ {u}\leq -k}  {u}( {u}+k) \biggr]\\
        &= M \biggr[ \int_{ {u}\geq k}  {u}{u} -  \int_{ {u}\geq k}  {u}k + \int_{ {u}\leq -k}  {u}{u} + \int_{ {u}\leq -k}  {u}k\biggr]\\
        &= M \biggr[\int_{A(k)}  {u} {u} -\int_{ {u}\geq k>0} \underbrace{ {u}k}_{>0} + \int_{ {u}\leq -k<0} \underbrace{ {u}k}_{<0} \biggr] \leq M \int_{A(k)}  {u} {u}.
    \end{aligned}
    \end{equation*}
\end{proof}

\end{document}